\documentclass[12pt]{amsart}

\usepackage[mathscr]{eucal}
\usepackage{setspace}
\usepackage[centertags]{amsmath}
\usepackage{amsfonts}
\usepackage{amssymb}
\usepackage{amsthm}
\usepackage{mathrsfs}
\usepackage{bbding}
\usepackage{bbold}
\usepackage{graphicx}
\usepackage{pifont}
\usepackage{verbatim}
\usepackage{eucal}
\usepackage{fancyhdr}
\usepackage{bm}
\usepackage{upgreek}
\usepackage{mathtools}
\usepackage{stmaryrd}
\usepackage [english]{babel}
\usepackage [autostyle, english = american]{csquotes}
\MakeOuterQuote{"}

\newtheorem{theorem}{Theorem}[section]
\newtheorem{lemma}[theorem]{Lemma}

\newtheorem{claim}[theorem]{Claim}
\newtheorem{corollary}[theorem]{Corollary}
\newtheorem{proposition}[theorem]{Proposition}

\newtheorem{cor}[theorem]{Corollary}

\theoremstyle{definition}
\newtheorem{definition}[theorem]{Definition}

\newtheorem{remark}[theorem]{Remark}

\advance \topmargin by -\headheight
\advance \topmargin by -\headsep
\evensidemargin \oddsidemargin
\newcommand{\ra}{\rightarrow}

\newcommand{\xra}{\xrightarrow}

\newcommand{\wt}[1]{\widetilde{#1}}
\newcommand{\mc}[1]{\mathcal{#1}}

\newcommand{\Es}[1]{\EuScript{#1}}

\newcommand{\cc}{\curvearrowright}

\numberwithin{equation}{section}
\DeclareSymbolFont{AMSb}{U}{msb}{m}{n}
\DeclareMathSymbol{\N}{\mathbin}{AMSb}{"4E}
\DeclareMathSymbol{\Z}{\mathbin}{AMSb}{"5A}
\DeclareMathSymbol{\R}{\mathbin}{AMSb}{"52}
\DeclareMathSymbol{\Q}{\mathbin}{AMSb}{"51}
\DeclareMathSymbol{\I}{\mathbin}{AMSb}{"49}
\DeclareMathSymbol{\B}{\mathbin}{AMSb}{"42}
\DeclareMathSymbol{\C}{\mathbin}{AMSb}{"43}
\DeclareMathSymbol{\T}{\mathbin}{AMSb}{"54}
\DeclareMathSymbol{\F}{\mathbin}{AMSb}{"46}
\DeclareMathSymbol{\PP}{\mathbin}{AMSb}{"50}

\renewcommand{\emptyset}{\varnothing}

\usepackage{mathtools}

\DeclarePairedDelimiter\abs{\lvert}{\rvert}%
\DeclarePairedDelimiter\norm{\lVert}{\rVert}%

\makeatletter
\let\oldabs\abs
\def\abs{\@ifstar{\oldabs}{\oldabs*}}
\let\oldnorm\norm
\def\norm{\@ifstar{\oldnorm}{\oldnorm*}}

\usepackage{mathrsfs}
\usepackage[all,cmtip]{xy}
\usepackage{graphicx}
\usepackage[centertags]{amsmath}
\usepackage{newlfont}
\usepackage{hyperref}

\usepackage[page,title,titletoc,header]{appendix}
\usepackage{stmaryrd}

\theoremstyle{plain}

\theoremstyle{definition}

\theoremstyle{remark}

\begin{document}
\pagestyle{headings}
\title{Superrigidity, measure equivalence, and Weak Pinsker entropy}
\author{Lewis Bowen and Robin D. Tucker-Drob}
\subjclass[2010]{Primary 37A20}

\begin{abstract}
We show that the class $\mathscr{B}$, of discrete groups which satisfy the conclusion of Popa's Cocycle Superrigidity Theorem for Bernoulli actions, is invariant under measure equivalence.

We generalize this to the setting of discrete p.m.p.\ groupoids, and as a consequence we deduce that any nonamenable lattice in a product of two noncompact, locally compact second countable groups, must belong to $\mathscr{B}$.

We also introduce a measure-conjugacy invariant called Weak Pinsker entropy and show that, if $G$ is a group in the class $\mathscr{B}$, then Weak Pinsker entropy is an orbit-equivalence invariant of every essentially free p.m.p.\ action of $G$.
\end{abstract}

\maketitle

\section*{Introduction}\label{sec:intro} Throughout this article all measure spaces are standard $\sigma$-finite measure spaces, i.e., standard Borel spaces equipped with a $\sigma$-finite Borel measure. We will suppress the given measure from our notation, e.g., stating that $X$, $Y$, and $Z$ are measure spaces. In this case we write $\mu _X$, $\mu _Y$, and $\mu _Z$ for the given measures on $X$, $Y$, and $Z$ respectively. We write $X\otimes Y$ for the independent product of the measure spaces $X$ and $Y$. We use the adjectives "Borel" and "measurable" interchangeably, with both meaning "Borel measurable."

\subsection{Superrigidity and measure equivalence} Two countable discrete groups $G$ and $H$ are said to be {\bf measure equivalent} if there exist commuting, essentially free measure preserving actions of $G$ and $H$ on a standard $\sigma$-finite nonzero measure space, such that each of the actions admits a measurable fundamental domain of finite measure. The example driving this definition, introduced by Gromov in \cite{Grom93}, is that of two groups $G$ and $H$ which are lattices in the same locally compact second countable group; in this situation, the measure equivalence is witnessed by the left and right translation actions, of $G$ and $H$ respectively, on the ambient locally compact group equipped with Haar measure.

Measure equivalence may also be characterized ergodic theoretically as follows \cite{Fu09}: $G$ and $H$ are measure equivalent if and only if there exist probability measure preserving (p.m.p.) actions $G\cc X$ and $H\cc Y$ whose translation groupoids are reduction equivalent (see Definition \ref{def:MEgroupoid}).

For a countable group $G$ and a probability space $K$, we let $G\cc K^G$ denote the standard Bernoulli action (a.k.a.\ Bernoulli shift) of $G$ with base $K$, i.e., the p.m.p.\ action of $G$ on $K^G$ given by $(gx)(h):= x(g^{-1}h)$ for $x\in K^G$, $g,h\in G$, where $K^G$ is equipped with product measure $\mu _{K^G}:=\mu _{K}^G$. We let $[0,1]$ denote the unit interval equipped with Lebesgue measure $\mu _{[0,1]}$.

\begin{definition}\label{def:BCS}
Let $\mathscr{C}$ be a class of Polish groups and let $G$ be a countable group.
\begin{itemize}
\item An extension $p: X\rightarrow Y$ of p.m.p.\ actions of $G$ is called {\bf relatively $\mathscr{C}$-superrigid} if every measurable cocycle $w:G\times X\rightarrow L$ taking values in a group $L\in \mathscr{C}$ is cohomologous to a cocycle which descends to $Y$.
\item The group $G$ is said to be {\bf Bernoulli $\mathscr{C}$-superrigid} if for every p.m.p.\ action $G\cc Y$ of $G$, the associated Bernoulli extension $Y\otimes [0,1]^G\rightarrow Y$ is relatively $\mathscr{C}$-superrigid.
\end{itemize}
\end{definition}

We are now ready to state the main theorem of the first part of this article.

\begin{theorem}\label{thm:main}
Let $\mathscr{C}$ be a class of Polish groups contained in the class $\mathscr{G}_{\mathrm{inv}}$ of Polish groups admitting a bi-invariant metric, and let $G$ and $H$ be countable groups which are measure equivalent. Then $G$ is Bernoulli $\mathscr{C}$-superrigid if and only if $H$ is Bernoulli $\mathscr{C}$-superrigid.
\end{theorem}

In practice, $\mathscr{C}$ will often be a subclass of the class $\mathscr{U}_{\mathrm{fin}}\subseteq \mathscr{G}_{\mathrm{inv}}$, of Polish groups which are isomorphic to a closed subgroup of the unitary group of a finite von Neumann algebra. This class was isolated by Popa in \cite{Pop07} as a class of target groups to which his deformation/rigidity techniques naturally apply, and it includes all countable discrete groups and all compact metrizable groups.

Conveniently, Definition \ref{def:BCS} was already anticipated by Popa in the groundbreaking articles \cite{Pop07, Pop08}. It follows directly from \cite[Theorem 0.1]{Pop07} and \cite[Theorem 4.1]{Pop08}, for example, that a group $G$ is Bernoulli $\mathscr{U}_{\mathrm{fin}}$-superrigid whenever it contains an infinite normal subgroup $N$ such that either (i) the pair $(G,N)$ has relative property (T), or (ii) $N$ is generated by (element-wise) commuting subgroups $H$ and $K$, with $H$ nonamenable and $K$ infinite. Theorem \ref{thm:main} now implies that any group which is measure equivalent to such a group $G$ is also Bernoulli $\mathscr{U}_{\mathrm{fin}}$-superrigid. For example, the simple groups constructed by Burger and Mozes in \cite{BM00}, being lattices in a product of automorphism groups of regular trees, are measure equivalent to a product of free groups, hence are Bernoulli $\mathscr{U}_{\mathrm{fin}}$-superrigid. More generally, by applying the groupoid version of Theorem \ref{thm:main} (Theorem \ref{thm:Gmain}), we obtain:

\begin{corollary}\label{cor:lattice}
Let $G=G_0\times G_1$ be a product of locally compact second countable groups, with $G_0$ nonamenable and $G_1$ noncompact. Then any lattice in $G$ is Bernoulli $\mathscr{U}_{\mathrm{fin}}$-superrigid.
\end{corollary}

\begin{remark}
Corollary \ref{cor:lattice} does not follow naively from Theorem \ref{thm:main} and Popa's Theorem, since even though $G$ contains a lattice, one of the groups $G_0$ or $G_1$ may fail to contain a lattice. This can be seen in the example $G_0 := \mathrm{SL}_n(\Z )\ltimes \R ^n$ and $G_1 := \mathrm{SL}_n(\Z _p)\ltimes \Q _p ^n$. The group $G_1$ does not contain a lattice, but $G_0\times G_1$ contains the lattice $\mathrm{SL}_n(\Z )\ltimes (\Z [\tfrac{1}{p}])^n$ via the diagonal embedding $\varphi _0 \times \varphi _1$, where $\varphi _i :\mathrm{SL}_n(\Z )\ltimes (\Z [\tfrac{1}{p}])^n\rightarrow G_i$ is the natural inclusion into $G_i$. However, as pointed out by P. Wesolek, if both $G_0$ and $G_1$ are totally disconnected locally compact groups, and $G_0\times G_1$ contains a lattice, then both $G_0$ and $G_1$ contain a lattice as well.
\end{remark}

Theorem \ref{thm:main} is related to, and largely inspired by, a theorem of Peterson and Sinclair stating that the class of groups $G$ whose group von Neumann algebra $LG$ is $L^2$-rigid, is closed under orbit equivalence \cite{PS11}. It also follows from \cite{PS11} and \cite{Pop07} that if $LG$ is $L^2$-rigid then $G$ is Bernoulli $\mathscr{U}_{\mathrm{fin}}$-superrigid. The converse is unclear; for example, a standard relativization of the proof of \cite[Theorem 11]{TD14} shows that any nonamenable inner amenable group is Bernoulli $\mathscr{U}_{\mathrm{fin}}$-superrigid, although it is unclear at the moment whether the group von Neumann algebra of any such group is necessarily $L^2$-rigid. (We note that it follows from \cite{Ki15}, \cite{Pet09}, and the aforementioned orbit equivalence invariance theorem from \cite{PS11}, that the group von Neumann algebra of Vaes's group from \cite{Va12} is $L^2$-rigid.) Part of the motivation for investigating Theorem \ref{thm:main} was to find a natural invariant of measure equivalence which applies to nonamenable inner amenable groups, and which implies cocycle superrigidity for Bernoulli actions.

While the results of \cite{Pop07,Pop08,PS11,TD14} provide many examples of Bernoulli $\mathscr{U}_{\mathrm{fin}}$-superrigid groups, the extent of this class of groups remains unclear. It follows from \cite{PS11} that the first $\ell ^2$-Betti number of a nonamenable group which is Bernoulli $\mathscr{U}_{\mathrm{fin}}$-superrigid must vanish. The converse -- whether every nonamenable group with vanishing first $\ell ^2$-Betti number is Bernoulli $\mathscr{U}_{\mathrm{fin}}$-superrigid -- is an open problem, a version of which appears in \cite{PopaProblems}. Note that, by \cite{Ga02}, vanishing of the first $\ell ^2$-Betti number is also an invariant of measure equivalence, so Theorem \ref{thm:main} at least puts these two properties on equal footing.

Many of the ideas going into the proof of Theorem \ref{thm:main} are already contained in A. Furman's article \cite{Fu07} and S. Popa's article \cite{Pop07}. In fact, not one of the three main ingredients involved in the proof of Theorem \ref{thm:main} is particularly new. Specifically, these main ingredients are: (1) the definition of Bernoulli $\mathscr{C}$-superrigidity, which was already implicitly considered in \cite{Pop07}; (2) Lemma \ref{lem:GFurman}, which is essentially already contained in \cite[Theorem 3.4]{Fu07}; and (3) Proposition \ref{prop:BernIsom}, which is more or less well known (and, in any case, trivial). The main new contribution of the first part of this article is the observation that these three ingredients play well together in order to produce Theorem \ref{thm:main}.

The most natural setting for the proof of Theorem \ref{thm:main} is that of discrete p.m.p.\ groupoids. However, we first present a "quick and dirty" direct proof of Theorem \ref{thm:main} for the special case of orbit equivalence in \S\ref{sec:OE}. The proof of Theorem \ref{thm:main} in general, given in \S\ref{sec:ME}, boils down to many of the same ideas, although the need to handle restrictions to positive measure sets makes it much more convenient to work in the groupoid context. This also has the benefit of producing Corollary \ref{cor:lattice}. We therefore gather the necessary background in \S\ref{sec:prelim}, and generalize Theorem \ref{thm:main} to the setting of discrete p.m.p.\ groupoids in Theorem \ref{thm:Gmain}, deducing Theorem \ref{thm:main} as a special case.

\subsection{When is entropy an orbit equivalence invariant?} The second part of this article explores consequences of Theorem \ref{thm:main} for entropy, orbit equivalence, and their interaction.

\subsubsection{What is entropy?} Kolmogorov and Sinai introduced entropy as a measure-conjugacy invariant for measure-preserving actions of the integers \cite{Kol58, Kol59, Sin59}. Kolmogorov showed that the entropy  of a Bernoulli shift action $\Z \cc K^{\Z}$ is the {\bf Shannon entropy} $H(K)$ of the base probability space $K$. The latter is defined by: if $\mu _K$ has countable support then
\[
H(K) := - \sum_{k \in K} \mu _K (\{k\}) \log (\mu _K (\{k\})).
\]
Otherwise $H(K):=+\infty$.

Later, Ornstein proved a converse: if $K$ and $L$ are probability spaces with the same Shannon entropy then the corresponding Bernoulli shifts $\Z \cc K^\Z$ and $\Z \cc L^\Z$ are measurably conjugate \cite{Orn70}. Building on work of the first author \cite{Bo11b}, Seward has recently generalized this result to arbitrary countably infinite groups \cite{Sew18}.

Sinai proved that an ergodic p.m.p.\ action $\Z \cc X$ of $\Z$ factors onto any Bernoulli shift $\Z \cc K^\Z$ for which $h(\Z \cc X ) \geq H(K)$ \cite{Sin64}. Since entropy of a $\Z$-action cannot increase under a factor, the condition $h(\Z \cc X)\geq H(K)$ is necessary. So entropy could have been defined as the supremum, over all Bernoulli shifts $\Z \cc K^\Z$ onto which the system factors, of the Shannon entropy $H(K)$ of the base space.

\subsubsection{Generalizations of classical entropy}
There are two main generalizations of classical entropy: sofic entropy and Rokhlin entropy. The sofic entropy of a p.m.p. action $G \cc X$ is defined in \cite{bowen-jams-2010, kerr-li-variational} (see also \cite{Bow17a, Bow17b} for an introduction and survey) in terms of counting approximations to the action relative to a given sofic approximation of the group. Rokhlin entropy was developed in a series of papers \cite{seward-kreiger-1, MR3959054}. The Rokhlin entropy of a p.m.p. action $G \cc X$ is always well-defined, unlike the sofic entropy which requires the group to be sofic (it is unknown whether all groups are sofic). In fact, the Rokhlin entropy of an action is the infimum of the Shannon entropies of generating partitions. By \cite{seward-kreiger-1}, Rokhlin entropy upper bounds sofic entropy. It is unknown whether they are equal (except when the sofic entropy is $-\infty$ in which case they cannot be equal). By \cite{seward-kreiger-1}, if $G$ is sofic then the sofic and Rokhlin entropies agree on Bernoulli shifts: they both equal the Shannon entropy of the base. However, if $G$ is not assumed to be sofic, then the Rokhlin entropy is only known to be upper-bounded by the Shannon entropy of the base. 

\subsubsection{Factors}

 In recent spectacular work, Seward generalized Sinai's Factor Theorem to all countable groups using Rokhlin entropy in place of Kolmogorov-Sinai entropy \cite{MR4066472}. This suggests another generalization of entropy as follows. Let $G \cc X$ be a p.m.p. action. Consider the supremum  over all Bernoulli shifts $G \cc K^G$ onto which $G\cc X$ factors, of the Shannon entropy $H(K)$ of the base space of the Bernoulli factor. By work of Ornstein-Weiss, this gives the usual notion of entropy when $G$ is amenable \cite{OW80}. However it fails badly when $G$ is non-amenable because then all Bernoulli shifts factor onto each other \cite{Bow17c}. To amend this situation, we restrict our attention to direct Bernoulli factors, as explained next.



\subsubsection{Weak Pinsker entropy}

Fix a countable discrete group $G$. Let $\alpha : G\cc X$ be an ergodic p.m.p.\ action of $G$ on a standard probability space $X$. If $\alpha$ is measurably conjugate to a direct product $\alpha \cong \alpha _0\otimes \alpha _1$, of two actions $\alpha _0$ and $\alpha _1$, then each of the actions $\alpha _0$ and $\alpha _1$ is called a {\bf direct factor} of $\alpha$. If $\beta$ is a direct factor of $\alpha$ which is measurably conjugate to a Bernoulli shift of $G$, then we say that $\beta$ is a {\bf direct Bernoulli factor} of $\alpha$. We define the {\bf Weak Pinsker entropy} $h^{\mathrm{WP}}(G\cc X)$ of the action $G\cc X$ to be the supremum, over all direct Bernoulli factors $G\cc K^G$ of $G\cc X$, of the base space Shannon entropy $H(K)$. Symbolically,
\[
h^{\mathrm{WP}}(G\cc X ) := \sup \{ H(K) \, : \, G\cc K^G \text{ is a direct Bernoulli factor of }G\cc X \} .
\]
It is immediate that Weak Pinsker entropy is invariant under measure conjugacy. Note that the trivial Bernoulli shift $G\cc \{ \ast \} ^G$, over a one point base space $\{ \ast \}$, is just the trivial action of $G$ on a one point space, and hence every action has the trivial one point action as a direct Bernoulli factor. In particular, $h^{\mathrm{WP}}(G\cc X )\geq 0$ for any ergodic p.m.p.\ action $G\cc X$. It follows immediately from the fact that the sofic entropy of a direct product of a Bernoulli shift with an arbitrary action is the sum of the sofic entropies \cite{bowen-jams-2010} that Weak Pinsker entropy is bounded from above by sofic entropy whenever the latter is not minus infinity. In particular, this implies that if $G$ is a sofic group, then the Weak Pinsker entropy of a Bernoulli shift is the Shannon entropy of its base.

 This definition is motivated by Thouvenot's problem, described next.

\subsubsection{The Weak Pinsker Property}
Suppose for now that $G$ is amenable. An ergodic action $\alpha$ of $G$ is said to have the {\bf Weak Pinsker Property} if for every $\epsilon >0$, the action $\alpha$ is measurably conjugate to a direct product $\alpha \cong \beta \otimes \alpha _0$, where $\beta$ is a Bernoulli shift of $G$, and the Kolmogorov-Sinai entropy of $\alpha _0$ is at most $\epsilon$. Since the Kolmogorov-Sinai entropy of a direct product action is the sum of the Kolmogorov-Sinai entropies of the corresponding direct factors, it follows that if $\alpha$ has the Weak Pinsker Property, then the Weak Pinsker entropy and Kolmogorov-Sinai entropy of $\alpha$ coincide.

Resolving a long-standing question due to Jean-Paul Thouvenot \cite{Thou77}, Tim Austin recently proved that all ergodic essentially free actions of amenable groups have the Weak Pinsker Property \cite{Aust17}. This is a major advance in classical entropy theory.
The Weak Pinsker Property can be generalized beyond the class of amenable groups by employing either sofic entropy or Rokhlin entropy in place of Kolmogorov-Sinai entropy. However, the paper \cite{bowen-wp} gives an example of an ergodic free group action that does not have the Weak Pinsker property (with respect to either generalization). In particular, for this action, the Weak Pinsker entropy is strictly less than both sofic and Rokhlin entropies. It would be interesting to obtain sufficient criteria for Weak Pinsker entropy to equal Rokhlin entropy, perhaps by determining when Seward's Factor Theorem \cite{MR4066472} produces a direct factor.

\subsubsection{Invariance under orbit equivalence} We let $\mathscr{G}_{\mathrm{dsc}}$ denote the class of all discrete countable groups. We propose the following ambiguous conjecture:

\medskip

\noindent {\bf Conjecture:} Whenever $G$ is Bernoulli $\mathscr{G}_{\mathrm{dsc}}$-superrigid, the entropy of any ergodic essentially free action of $G$ is an orbit-equivalence invariant.

\medskip

\noindent The ambiguity here is in the notion of entropy. The conjecture remains open for sofic and Rokhlin entropy. However, we prove this conjecture for Weak Pinsker entropy below. More generally, we prove that under a stable orbit equivalence, the expected scaling occurs. Recall that two ergodic essentially free p.m.p.\ actions $G\cc X$ and $H\cc Y$ of countable groups $G$ and $H$ are said to be {\bf stably orbit equivalent} if there exist positive measure sets $A\subseteq X$ and $B\subseteq Y$ such that the groupoids $(G\ltimes X)_A$ and $(H\ltimes Y)_B$ are isomorphic.
%
%

\begin{theorem}\label{thm:soe} Let $G\cc X$ and $H \cc Y$ be ergodic essentially free p.m.p.\ actions of countable groups $G$ and $H$, and assume that $G$ is Bernoulli $\mathscr{G}_{\mathrm{dsc}}$-superrigid. Suppose that the actions are stably orbit equivalent, so that there exist positive measure sets $A\subseteq X$ and $B\subseteq Y$ such that $(G\ltimes X)_A$ and $(H\ltimes Y)_B$ are isomorphic. Then
\[
h^{\mathrm{WP}}(G\cc X ) = \frac{\mu _X (A)}{\mu _Y (B)} h^{\mathrm{WP}}(H\cc Y ) .
\]
\end{theorem}
Theorem \ref{thm:soe} suggests a way of defining Weak Pinsker entropy directly and intrinsically for an ergodic discrete p.m.p.\ groupoid, as we now describe.

\subsubsection{Groupoid Weak Pinsker entropy}\label{subsec:gwpe} Given an ergodic discrete p.m.p.\ groupoid $\mc{G}$, we define the {\bf groupoid weak Pinsker entropy} of $\mc{G}$, denoted $h^{\mathrm{gWP}}(\mc{G})$, to be the supremum of $H(K)$, where the supremum is taken over all pairs $(\mc{H}, K)$ for which there exists an isomorphism between $\mc{G}$ and the translation groupoid $\mc{H}\ltimes K^{\otimes \mc{H}}$ associated to the Bernoulli action of $\mc{H}$ with base $K$ (see section \ref{sec:prelim} for the definition of Bernoulli actions of groupoids):
\[
h^{\mathrm{gWP}}(\mc{G}) := \sup \{ H(K) \, : \, \mc{G}\cong \mc{H}\ltimes K^{\otimes \mc{H}} \text{ for some discrete p.m.p.\ groupoid }\mc{H} \} .
\]
It is clear from the definition that $0\leq h^{\mathrm{gWP}}(\mc{G})\leq +\infty$. What is not obvious is whether $h^{\mathrm{gWP}}(\mc{G})$ can take values other than $0$ and $+\infty$. For example, if $\mc{G}$ is periodic, or if $\mc{G}$ is not principal, then $h^{\mathrm{gWP}}(\mc{G})=0$, since any ergodic groupoid which can be expressed nontrivially as a Bernoulli extension must be (by ergodicity) aperiodic and principal. On the other hand, if we let $\mc{R}_0$ denote the ergodic hyperfinite type-$\mathrm{II}_1$ equivalence relation, then we have $h^{\mathrm{gWP}}(\mc{R}_0)=+\infty$, since $\mc{R}_0$ is isomorphic to the translation groupoid $\Z \ltimes [0,1]^{\Z}$ associated to the infinite entropy Bernoulli shift of $\Z$. The fact that $h^{\mathrm{gWP}}(\mc{G})$ can take finite nonzero values is a consequence of Theorem \ref{thm:gWPvsWP} below, which shows that the groupoid Weak Pinsker entropy of a translation groupoid associated to an ergodic action of a Bernoulli $\mathscr{G}_{\mathrm{dsc}}$-superrigid group coincides with the Weak Pinsker entropy of the action.



\medskip

\noindent {\bf Acknowledgements:} LB thanks Daniel Drimbe for an inspiring conversation. RTD would like to thank Adrian Ioana and Phillip Wesolek for valuable feedback and stimulating discussions. LB was supported by NSF grant DMS-1500389 and a Simons Fellowship. RTD was supported by NSF grant DMS-1855825.

\begin{remark}\label{rem:null}
Throughout the article, when working in any measure theoretic context, such as that of discrete p.m.p.\ groupoids, we are only interested in properties which are not sensitive to changes on a null set. With this in mind, we will identify two sets when their symmetric difference has measure zero, we will identify two maps when their domains and values agree almost everywhere, and we will often ignore null sets when it is convenient and appropriate.
\end{remark}

\setcounter{tocdepth}{4}
\tableofcontents

\section{The case of orbit equivalence}\label{sec:OE}
Recall that two countable groups $G$ and $H$ are said to be {\bf orbit equivalent} if they admit free ergodic p.m.p.\ actions $G\cc X$ and $H\cc Y$ which generate isomorphic orbit equivalence relations, i.e., for which the translation groupoids $G\ltimes X$ and $H\ltimes Y$ are isomorphic.

Let $G\cc X$ and $G\cc Y$ be two p.m.p.\ actions of a countable group $G$. An {\bf extension} of p.m.p.\ actions of $G$, from $G\cc X$ to $G\cc Y$, is a measure preserving map $q:X\rightarrow Y$ such that $q(gx)=gq(x)$ for all $g\in G$ and a.e.\ $x\in X$. If $X_0\rightarrow Y$ and $X_1\rightarrow Y$ are two extensions of p.m.p.\ actions of $G$ then we consider the p.m.p.\ action of $G$ on the relatively independent product $X_0\otimes _Y X_1$ (see \S \ref{subsec:fibmeas}), given by $g(x_0,x_1):=(gx_0,gx_1)$. This action is then naturally viewed as an extension $X_0\otimes _Y X_1\rightarrow Y$, of $Y$.

An extension $q:X\rightarrow Y$ of p.m.p.\ actions of $G$ is called {\bf relatively ergodic} if for every $G$-invariant measurable subset $A$ of $X$, there is some measurable subset $B$ of $Y$ such that $\mu _X (A\triangle q^{-1}(B) ) =0$. The extension $q$ is called {\bf relatively weakly mixing} if the relatively independent extension $X\otimes _{Y} X \rightarrow Y$ is relatively ergodic.

We will need the following "asymmetric" generalization of the cocycle untwisting result \cite[Theorem 3.4]{Fu07}, \cite[Theorem 3.1]{Pop07}. A complete proof is provided (in groupoid language) in Lemma \ref{lem:GFurman} below.

\begin{lemma}[Untwisting lemma, asymmetric version]\label{lem:Furman}
Let $q_0:X_0\rightarrow Y$ and $q_1 : X_1\rightarrow Y$ be two extensions of a p.m.p.\ action $G\cc Y$ of $G$, and assume that the extension $q_1:X_1\rightarrow  Y$ is relatively weakly mixing. Let $L$ be a Polish group with a bi-invariant metric, and let $w_0 : G\times X_0 \rightarrow L$ and $w_1 :G\times X_1 \rightarrow L$ be measurable cocycles. Let $X= X_0\otimes _Y X_1$ be the relatively independent joining over $Y$. Suppose that $F:X\rightarrow L$ is a measurable map satisfying, for $\mu _X$-a.e.\ $(x_0,x_1)\in X$ and all $g\in G$,
\begin{equation}\label{eqn:Fw}
w_0(g,x_0) = F(gx_0,gx_1)w_1(g,x_1)F(x_0,x_1)^{-1}.
\end{equation}
Then there exist measurable maps $\varphi _0:X_0\rightarrow L$, $\varphi _1:X_1\rightarrow L$, and a measurable cocycle $w:G\times Y \rightarrow L$ such that for $\mu _X$-a.e. $(x_0,x_1)\in X$ and all $g\in G$ we have
\begin{align*}
F(x_0,x_1)&=\varphi _0(x_0)^{-1}\varphi _1(x_1), \\
w (g,q_0(x_0)) &= \varphi _0(gx_0) w_0(g,x_0)\varphi _0 (x_0)^{-1},\\
w (g,q_1(x_1)) &= \varphi _1 (gx_1)w_1(g,x_1)\varphi _1(x_1)^{-1} .
\end{align*}
\end{lemma}

\begin{proof}[Proof sketch] The assumption that $q_1$ is relatively weakly mixing ensures that the extension $(X_0\otimes _Y X_0 )\otimes _Y X_1 \rightarrow X_0\otimes _Y X_0$ is relatively weakly mixing. We then define the map $\Phi : (X_0\otimes _Y X_0)\otimes _Y X_1 \rightarrow L$ by $\Phi ((x_0,x_0'),x_1) := F(x_0,x_1)F(x_0',x_1)^{-1}$. The rest of the proof of Lemma \ref{lem:Furman} is exactly analogous to that of Theorem 3.4 of \cite{Fu07}.
\end{proof}

We can now prove that, for $\mathscr{C}\subseteq \mathscr{G}_{\mathrm{inv}}$, Bernoulli $\mathscr{C}$-superrigidity is an orbit equivalence invariant:

\begin{proof}[Proof of Theorem \ref{thm:main} in the case of orbit equivalence]
Suppose that $G$ and $H$ are orbit equivalent. Assuming that $H$ is Bernoulli $\mathscr{C}$-superrigid, we must show that $G$ is Bernoulli $\mathscr{C}$-superrigid as well. Toward this goal, let $G\cc Y$ be a p.m.p.\ action of $G$, let $L\in \mathscr{C}$, and let $w_1:G\times (Y\otimes [0,1]^G) \rightarrow L$ be a measurable cocycle for the product action $G\cc Y\otimes [0,1]^G$. We must show that $w_1$ is cohomologous to a cocycle which descends to $Y$. This is clear if $G$ is finite, so we may assume that $G$ is infinite.

Since $G$ and $H$ are orbit equivalent we may find free ergodic p.m.p.\ actions $G\cc Z$ and $H\cc Z$ which generate the same orbit equivalence relation. Let $u:H\times Z\rightarrow G$ and $v:G\times Z \rightarrow H$ be the associated rearrangement cocycles, so that $u(h,z)z = hz$ and $v(g,z)z=gz$ for all $g\in G$, $h\in H$, $z\in Z$. Let $X_0:= Z\otimes Y$, let $G\cc X_0$ be the diagonal product action, and let $H\cc X_0$ be the action defined by $h\cdot (z,y)= (hz,u(h,z)y)$. These actions of $G$ and $H$ then generate the same orbit equivalence relation on $X_0$.

Let $X_G:= X_0\otimes [0,1]^G$ and let $G\cc X_G$ be the diagonal product of $G\cc X_0$ with the standard Bernoulli action of $G$. Likewise, let $X_H:= X_0\otimes [0,1]^H$, and let $H\cc X_H$ be the diagonal product of $H\cc X_0$ with the standard Bernoulli action of $H$. There is then a natural orbit equivalence between these actions of $G$ and $H$, which we now describe (and which is essentially the same fact underlying \cite[Proposition 3.2]{BHI15} and Proposition \ref{prop:BernIsom} below). For each $z\in Z$ define $\Phi _z: [0,1]^G \rightarrow [0,1]^H$ by $\Phi _z (b)(h) = b(u(h, h^{-1}z))$, and define $\Phi : X_G \rightarrow X_H$ by $\Phi (z,y,b) := (z,y,\Phi _z(b))$.

\begin{claim}\label{claim:OE} $\Phi$ is an orbit equivalence.
\end{claim}

\begin{proof}[Proof of Claim \ref{claim:OE}] Since for each $z\in Z$ the map $h\mapsto u(h,h^{-1}z )$ is a bijection from $H$ to $G$, it follows that $\Phi _z$ is a measure space isomorphism from $[0,1]^G$ to $[0,1]^H$, and therefore $\Phi$ is a measure space isomorphism from $X_G$ to $X_H$. Moreover, the cocycle identity implies that $\Phi _{gz}(gb) = v(g,z)\Phi _z(b)$ and hence $\Phi (g(z,y,b)) = v(g,z)\Phi (z,y,b)$ for all $g\in G$ and $(z,y,b)\in X_G$. Since for each $z\in Z$ the map $g\mapsto v(g,z)$ is a bijection from $G$ to $H$, it follows that $\Phi$ maps the $G$-orbit of each $(z,y,b) \in X_G$ bijectively onto the $H$-orbit of $\Phi (z,y,b)\in X_H$, and hence $\Phi$ is an orbit equivalence. \qedhere[Claim \ref{claim:OE}]
\end{proof}

Let $\tilde{w}_1 : G\times X_G \rightarrow L$ be the lift of the cocycle $w_1$ to $X_G$, i.e., $\tilde{w}_1(g,(z,y,b)) := w_1(g,(y,b))$. Then the map $\tilde{w}_1^{\Phi} : H\times X_H \rightarrow L$, defined by
\[
\tilde{w}_1^{\Phi}(h,(z,y,c)):= \tilde{w}_1(u(h,z), \Phi ^{-1}(z,y,c))= w_1(u(h,z), (y,\Phi _z^{-1}(c))),
\]
is a cocycle of $H\cc X_H$. Since $H$ is Bernoulli $\mathscr{C}$-superrigid, there exists a map $F':X_H\rightarrow L$ and a cocycle $w_0' : H\times X_0 \rightarrow L$ such that for a.e.\ $(z,y,c)\in X_H$ and all $h\in H$ we have
\[
F'(h(z,y,c))w_0' (h, (z,y))F'(z,y,c)^{-1} = \tilde{w}_1^{\Phi}(h,(z,y,c)) ,
\]
i.e.,
\begin{equation}\label{eqn:hw0}
F'(h(z,y,c))w_0'(h, (z,y))F'(z,y,c)^{-1} = w_1(u(h,z), (y,\Phi _z^{-1}(c))) .
\end{equation}
Let $F= F'\circ \Phi$ and define the cocycle $w_0:G\times X_0 \rightarrow L$ by $w_0(g,(z,y)) := w_0' (v(g,z), (z,y))$.  Then for a.e.\ $(z,y,b)\in  X_G$, for all $g\in G$, by applying \eqref{eqn:hw0} to $h:= v(g,z)$ (so that $g=u(h,z)$) and $c:= \Phi _z(b)$ we obtain
\begin{equation}\label{eqn:gw0}
F(g(z,y,b))w_0(g,(z,y))F(z,y,b)^{-1} = w_1(g,(y,b)).
\end{equation}
Let $X_1:= Y\otimes [0,1]^G$. Since $G$ is infinite, the Bernoulli action $G\cc [0,1]^G$ is weakly mixing, and hence the extension $X_1\rightarrow  Y$ is relatively weakly mixing. It now follows from \eqref{eqn:gw0} and Lemma \ref{lem:Furman} that $w_1$ is cohomologous to a cocycle which descends to $Y$. This completes the proof.
\end{proof}

\section{Groupoid preliminaries}\label{sec:prelim}

\subsection{Discrete p.m.p.\ groupoids}
Given a groupoid $\mc{G}$, we denote its unit space by $\mc{G}^0$, and its source and range maps by $s_{\mc{G}}:\mc{G}\rightarrow \mc{G}^0$ and $r_{\mc{G}}:\mc{G}\rightarrow \mc{G}^0$ respectively, or simply by $s$ and $r$ when $\mc{G}$ is clear from the context. We always view $\mc{G}^0$ as a subset of $\mc{G}$, so that $\mc{G}^0 = \{ g\in \mc{G} \, : \, s(g)=g=r(g) \}$, and we will often denote elements of $\mc{G}^0$ by the letters $x$, $y$, and $z$. Given subsets $A,B\subseteq \mc{G}$ we let $AB = \{ gh \, : \, g\in A , \ h\in B , \, s(g)=r(h) \}$, and we let $A^{-1}= \{ g^{-1} \, : \, g\in A \}$. Given $g\in \mc{G}$ and $A\subseteq \mc{G}$ we write $gA$ for $\{ g \} A$. Thus, for $x\in \mc{G}^0$ we have $\mc{G}x=s^{-1}(x)$ and $x\mc{G}=r^{-1}(x)$.

Given $A\subseteq \mc{G}^0$, we let $\mc{G}_A:=A\mc{G}A$, so that $\mc{G}_A$ is itself a groupoid, called the {\bf reduction} of $\mc{G}$ to $A$, with unit space $\mc{G}_A^0:=A$, and groupoid operations inherited from $\mc{G}$. A {\bf complete unit section} of $\mc{G}$ is a subset $A$ of $\mc{G}^0$ satisfying $\mc{G}A\mc{G}= \mc{G}$; equivalently, $A\subseteq \mc{G}^0$ is a complete unit section of $\mc{G}$ if and only if $A$ meets every equivalence class of the equivalence relation
\[
\mc{R}_{\mc{G}}:= \{ (r(g),s(g)) \, : \, g\in \mc{G} \}
\]
on $\mc{G}^0$, associated to $\mc{G}$. The groupoid $\mc{G}$ is said to be {\bf principal} if the map $\mc{G}\rightarrow \mc{R}_{\mc{G}}$, $g\mapsto (r(g),s(g))$, is injective, i.e., $\mc{G}$ is principal if $\mc{G}$ is isomorphic to an equivalence relation. A subset $A$ of $\mc{G}^0$ is {\bf $\mc{G}$-invariant} if it is $\mc{R}_{\mc{G}}$-invariant, i.e., if $A$ is a union of $\mc{R}_{\mc{G}}$-classes. A subset $\phi$ of $\mc{G}$ is called a {\bf bisection} of $\mc{G}$ if the source and range maps are both injective on $\phi$. If $\phi$ is a bisection of $\mc{G}$, then for each $g\in \mc{G}r(\phi )$ (respectively: $g\in s(\phi ) \mc{G})$ the set $g\phi$ (respectively: $\phi g$) consists of a single element of $\mc{G}$, and by abuse of notation we will also denote this element by $g\phi$ (respectively: $\phi g$). Likewise, if $g\in \mc{G}_{r(\phi )}$ then we identify $\phi ^{-1} g\phi$ with an element of $\mc{G}_{s(\phi )}$; the map $g\mapsto \phi ^{-1}g\phi$ is then a groupoid isomorphism from $\mc{G}_{r(\phi )}$ to $\mc{G}_{s(\phi )}$.

Let $\mc{H}$ and $\mc{G}$ be groupoids. A groupoid homomorphism $q:\mc{H}\rightarrow \mc{G}$ is said to be {\bf locally bijective} if for each $y\in \mc{H}^0$ the restriction $q_y: \mc{H}y\rightarrow \mc{G}q(y)$, of $q$ to $\mc{H}y$, is a bijection from $\mc{H}y$ to $\mc{G}q(y)$. Since for each $y\in \mc{H}$ the inverse map $h\mapsto h^{-1}$ provides a bijection from $\mc{H}y$ to $y\mc{H}$, this is equivalent to requiring that, for each $y\in \mc{H}^0$, the restriction $q^y:y\mc{H}\rightarrow q(y)\mc{G}$, of $q$  to $y\mc{H}$, is a bijection from $y\mc{H}$ to $q(y)\mc{G}$. 
If $p:\mc{H}\rightarrow \mc{G}$ is a groupoid homomorphism then we let $p^0 : \mc{H}^0\rightarrow \mc{G}^0$ denote the restriction of $p$ to $\mc{H}^0$.

A {\bf discrete Borel groupoid} is a groupoid $\mc{G}$, equipped with the structure of a standard Borel space, such that $\mc{G}^0$ is a Borel subset of $\mc{G}$, the source and range maps $s$ and $r$ are Borel and countable-to-one, and the multiplication and inversion maps are both Borel. In the category of discrete Borel groupoids, an isomorphism from $\mc{G}$ to $\mc{H}$ is a groupoid isomorphism $\varphi :\mc{G}\rightarrow \mc{H}$, from $\mc{G}$ to $\mc{H}$, which is Borel (i.e., $\varphi$ is a Borel map). We will often make use of the {\bf Lusin-Novikov Uniformization Theorem} \cite[Theorem 18.10]{Ke95}, which implies that if $f:A\rightarrow Y$ is a countable-to-one Borel function from a Borel subset $A$ of a standard Borel space $X$ into a standard Borel space $Y$, then $f(A)$ is Borel and there is a countable partition of $A$ into Borel sets on each of which $f$ is in injective.

A {\bf discrete p.m.p.\ groupoid} is a discrete Borel groupoid $\mc{G}$, along with a Borel probability measure $\mu _{\mc{G}^0}$ on $\mc{G}^0$ satisfying
\begin{equation}\label{eqn:pmp}
\int _{\mc{G}^0}|xD| \, d\mu _{\mc{G}^0}(x) = \int _{\mc{G}^0}|Dx| \, d\mu _{\mc{G}^0}(x)
\end{equation}
for all Borel subsets $D$ of $\mc{G}$, where $|D|$ denotes the cardinality of a set $D$. We let $\mu _{\mc{G}}$ denote the associated $\sigma$-finite Borel measure on $\mc{G}$, i.e., with $\mu _{\mc{G}}(D)$ given by \eqref{eqn:pmp} for $D\subseteq \mc{G}$ Borel.  A discrete p.m.p.\ groupoid $\mc{G}$ is called {\bf aperiodic} if $x\mc{G}$ is infinite for a.e.\ $x\in \mc{G}^0$, and $\mc{G}$ is called {\bf periodic} if $x\mc{G}$ is finite for a.e.\ $x\in \mc{G}^0$. We say that $\mc{G}$ is {\bf ergodic} if every $\mc{G}$-invariant Borel subsets of $\mc{G}^0$ is either $\mu _{\mc{G}^0}$-null or $\mu _{\mc{G}^0}$-conull. We will frequently make use of the fact that if $\mc{G}$ is an ergodic discrete p.m.p.\ groupoid, and if $A$ and $B$ are Borel subsets of $\mc{G}^0$ having the same measure then, after discarding a null set, there exists a Borel bisection $\theta$ of $\mc{G}$ with $s(\theta )=A$ and $r(\theta ) =B$; this follows from the case of principal groupoids (see \cite[Lemma 7.10]{KM04}) by the Lusin-Novikov Uniformization Theorem.

We call a measure preserving groupoid homomorphism $\mc{G}\xra{p}\mc{H}$ between discrete p.m.p.\ groupoids $\mc{G}$ and $\mc{H}$ a {\bf groupoid extension}. Equivalently, a groupoid extension from $\mc{G}$ to $\mc{H}$ is a Borel homomorphism $p$ from $\mc{G}$ to $\mc{H}$ such that (i) $p$ takes $\mu _{\mc{G}^0}$ to $\mu _{\mc{H}^0}$ and (ii) $p$ is locally bijective a.e., i.e., $p$ maps $x\mc{G}$ bijectively onto $p(x)\mc{H}$ for $\mu _{\mc{G} ^0}$-a.e.\ $x\in \mc{G}^0$. As indicated by Remark \ref{rem:null}, we will identify two extensions $\mc{G}\xra{p}\mc{H}$ and $\mc{G}\xra{q}\mc{H}$ if they agree on a $\mu _{\mc{G}}$-conull set. Section \ref{sec:ext} discusses extensions in more detail.

An {\bf isomorphism} from the discrete p.m.p.\ groupoid $\mc{G}$ to the discrete p.m.p.\ groupoid $\mc{H}$ is a measure preserving Borel groupoid isomorphism between conull subgroupoids $\mc{G}'\subseteq \mc{G}$ and $\mc{H}'\subseteq \mc{H}$. Given such an isomorphism from $\mc{G}$ to $\mc{H}$, after discarding null sets from $\mc{G}'$ and $\mc{H}'$ we can always assume that $\mc{G}' = \mc{G}_A$ and $\mc{H}'=\mc{H}_B$ for some $\mu _{\mc{G}^0}$-conull subset $A$ of $\mc{G}^0$ and some $\mu _{\mc{H}^0}$-conull subset $B$ of $\mc{H}^0$. Since we identify maps which agree a.e., this definition of isomorphism coincides with isomorphism in the category $\bm{\mathrm{DPG}}$, whose objects are discrete p.m.p.\ groupoids, and whose morphisms are given by (equality-a.e.\ equivalence classes of) groupoid extensions.
%

\subsection{Fibered Borel spaces} Let $X$ be a standard Borel space.

A {\bf fibered Borel space} over $X$ is a standard Borel space $Z$, along with a Borel map $p:Z\rightarrow X$ from $Z$ to $X$. We write $Z_x$ for the fiber $p^{-1}(x)$ over $x\in X$. We will often leave the fibering map implicit, e.g., stating that $Z$ is a fibered Borel space over $X$, or that $Z\rightarrow X$ is a fibered Borel space. A fibered Borel space $Z$ is called {\bf discrete} if each fiber is countable. If $\mc{G}$ is a discrete Borel groupoid then the source and range maps each make $\mc{G}$ into a discrete fibered Borel space over $\mc{G}^0$.

Let $\mc{Y}$ be a countable collection of fibered Borel spaces over $X$. Then the {\bf fibered product} over $X$ of the collection $\mc{Y}$, denoted $\bigotimes _X \mc{Y}$, is the fibered Borel space whose fiber over $x\in X$ is the direct product $(\bigotimes _X \mc{Y} )_x := \prod \{ Y_x \, : \, Y\in \mc{Y} \}$. We equip $\bigotimes _X \mc{Y}$ with the standard Borel structure which it inherits as a Borel subset of the direct product $\prod \mc{Y}$. In the case where $\mc{Y}= \{ Y, Z \}$ consists of two spaces, we denote the associated fibered product by $Y\otimes _X Z$.

If $\mc{G}$ is a discrete Borel groupoid and $Y$ is a fibered Borel space over $\mc{G}^0$ then we let $\mc{G}\otimes _{\mc{G}^0}Y$ denote the fibered product with respect to the fibering $s:\mc{G}\rightarrow \mc{G}^0$, so that $\mc{G}\otimes _{\mc{G}^0}Y = \{ (g,y) \in \mc{G}\times Y \, : \, y\in Y_{s(g)} \}$.
\subsection{Fibered measure spaces}\label{subsec:fibmeas} Let $X$ be a standard $\sigma$-finite measure space.

A {\bf fibered measure space} over $X$ is a fibered Borel space $Z$ over $X$, together with an assignment, $x\mapsto \mu _{Z_x}$, where
\begin{itemize}
\item[(i)] $\mu _{Z_x}$ is a $\sigma$-finite Borel measure on $Z$ which concentrates on $Z_x$ for each $x\in X$, and
\item[(ii)] the map $x\mapsto \mu _{Z_x}(B)$ is Borel whenever $B\subseteq Z$ is Borel.
\end{itemize}
If $Z\rightarrow X$ is a fibered measure space over $X$ then we will naturally consider $Z$ itself as a measure space by equipping it with the measure $\mu _Z:=\int _X \mu _{Z_x}\, d\mu _X$. A fibered measure space $Z$ over $X$ is called {\bf discrete} if each fiber $Z_x$ is countable and $\mu _{Z_x}$ is counting measure on $Z_x$. For example, a discrete p.m.p.\ groupoid $\mc{G}$ is naturally a discrete fibered measure space over $\mc{G}^0$ with respect to each of the fibering maps $s:\mc{G}\rightarrow \mc{G}^0$ and $r:\mc{G}\rightarrow \mc{G}^0$. In this case, both fiberings induce the same measure on $\mc{G}$, namely $\mu _{\mc{G}}$.

A {\bf fibered probability space} over $X$ is a fibered measure space $Z$ over $X$ in which $\mu _{Z_x}$ is a probability measure for a.e.\ $x\in X$; in this case the fibering map is measure preserving, taking $\mu _Z$ to $\mu _X$. Conversely, by the measure disintegration theorem, if $p:Z\rightarrow X$ is a measure preserving map between standard $\sigma$-finite measure spaces $Z$ and $X$, then there is an essentially unique integral representation, $\mu _Z = \int _X \mu _{Z_x}\, d\mu _X$, of $\mu _Z$ that makes $Z$ a fibered probability space over $X$.

Let $Y$ and $Z$ be fibered measure spaces over $X$. Then $Y\otimes _X Z$ is naturally a fibered measure space over $X$, where the fiber $(Y\otimes _X Z )_x = Y_x \times Z_x$ over $x\in X$ is equipped with the product measure $\mu _{(Y\otimes _X Z)_x} := \mu _{Y_x}\otimes \mu _{Z_x}$. In this setting, we call $Y\otimes _X Z$ the {\bf relatively independent product} over $X$ of $Y$ and $Z$. Let $\mc{Y}$ be a countable collection of fibered probability spaces over $X$. Then $\bigotimes _X \mc{Y}$ is naturally a fibered probability space over $X$, where the fiber $(\bigotimes _X \mc{Y})_x = \prod \{ Y_x \, : \, Y\in \mc{Y} \}$ over $x\in X$ is equipped with the product measure $\mu _{(\bigotimes _X\mc{Y})_x} := \prod \{ \mu _{Y_x} \, : \, Y\in \mc{Y} \}$. We call the resulting fibered probability space $\bigotimes _X \mc{Y}$ the {\bf relatively independent product} over $X$ of the collection $\mc{Y}$.

Let $Y$ and $Z$ be two fibered measure spaces over $X$. A {\bf fiberwise measure preserving map} over $X$ from $Y$ to $Z$ is a Borel map $\varphi :Y\rightarrow Z$ satisfying $\varphi _*\mu _{Y_x}= \mu _{Z_x}$ for a.e.\ $x\in X$. We identify two such maps if they agree on a $\mu _Y$-conull set. We call such a map $\varphi$ a {\bf fiberwise isomorphism} over $X$ if there is a $\mu _Y$-conull subset of $Y$ on which $\varphi$ is injective.

\subsection{Actions of groupoids} Let $\mc{G}$ be a discrete Borel groupoid. A {\bf Borel action} of $\mc{G}$ consists of a fibered Borel space $Y$ over $\mc{G}^0$, along with an assignment, $g\mapsto \alpha (g)$, of a Borel isomorphism $\alpha (g) : Y_{s(g)}\rightarrow Y_{r(g)}$ to each $g\in \mc{G}$, such that $\alpha (g)\alpha (h)=\alpha (gh)$ whenever $s(g)=r(h)$, and the associated map $(g,y)\mapsto \alpha (g)y$ is Borel from $\mc{G}\otimes _{\mc{G}^0}Y$ to $Y$. We denote such an action by $\alpha$, or $\alpha : \mc{G}\cc Y$, or $\mc{G}\cc Y$, depending on the context. 

Suppose now that $\mc{G}$ is a discrete p.m.p.\ groupoid. A {\bf measure preserving action} of $\mc{G}$ is a Borel action $\alpha :\mc{G}\cc Y$, in which $Y$ is a fibered measure space over $\mc{G}^0$, and for each $g\in \mc{G}$ the transformation $\alpha (g) : Y_{s(g)}\rightarrow Y_{r(g)}$ is a measure space isomorphism. We call such an action $\mc{G}\cc Y$ a {\bf discrete action} of $\mc{G}$ if $Y$ is a discrete fibered measure space over $\mc{G}^0$, and we call $\mc{G}\cc Y$ a {\bf p.m.p.\ action} of $\mc{G}$ if $Y$ is a fibered probability space over $\mc{G}^0$.

Two measure preserving actions $\alpha : \mc{G}\cc Y$ and $\beta : \mc{G}\cc Z$ of $\mc{G}$ are {\bf isomorphic}, denoted $\alpha \cong \beta$, if there exists a fiberwise isomorphism $\varphi :Y\rightarrow Z$ over $\mc{G}^0$ which is $\mc{G}$-equivariant, i.e., which satisfies $\varphi (\alpha (g)y)=\beta (g)\varphi (y)$ for a.e.\ $(g,y)\in \mc{G}\otimes _{\mc{G}^0}Y$.

The {\bf product} of two measure preserving actions $\alpha : \mc{G}\cc Y$ and $\beta : \mc{G}\cc Z$ of $\mc{G}$ is the measure preserving action $\alpha \otimes \beta : \mc{G}\cc Y\otimes _{\mc{G}^0}Z$ defined by
\[
(\alpha \otimes \beta )(g):=\alpha (g)\otimes \beta (g): Y_{s(g)}\otimes Z_{s(g)}\rightarrow Y_{r(g)}\otimes Z_{r(g)}
\]
for each $g\in \mc{G}$, i.e., $(\alpha \otimes \beta )(g)(y,z) = (\alpha (g)y,\beta (g)z)$ for $(y,z)\in Y_{s(g)}\times Z_{s(g)}$. The product of a countable collection $\{ \alpha _i :\mc{G}\cc Y_i \} _{i \in I}$ of p.m.p.\ actions of $\mc{G}$ is the p.m.p.\ action $\bigotimes _{i\in I} \alpha _i : \mc{G}\cc \bigotimes _{\mc{G}^0} \{ Y_i\, : \, i\in I \}$, defined by $(\bigotimes _{i\in I}\alpha _i )(g) := \prod \{ \alpha _i (g) \, : \, i \in I \}$ for each $g\in \mc{G}$, i.e., $(\bigotimes _{i\in I}\alpha _i )(g)$ maps $(y_i)_{i\in I} \in \prod \{ (Y_i)_{s(g)} \, : \, i \in I \}$ to $(\alpha _i (g)y_i )_{i\in I}\in \prod \{ (Y_i)_{r(g)} \, : \, i \in I \}$.

\subsection{Bernoulli actions of groupoids}
Let $X$ be a standard probability space and let $V$ be a discrete fibered measure space over $X$. Given another standard probability space $K$, we define the fibered probability space $K^{\otimes V}$ over $X$ as follows: for each $x\in X$, the fiber $K^{\otimes V}_x$ over $x$ is the product space $K^{\otimes V}_x := K^{V_x}$, equipped with the product measure $\mu _{K^{\otimes V}_x} := \mu _K ^{V_x}$. We equip $K^{\otimes V}:= \bigsqcup _{x\in X}K^{V_x}$ with the $\sigma$-algebra generated by the fibering map $p: K^{\otimes V}\rightarrow X$ along with all maps $K^{\otimes V}\rightarrow K$ of the form $f\mapsto f(t(p(f)))$, where $t:X\rightarrow V$ ranges over all Borel sections of $V\rightarrow X$. As a consequence of the Lusin-Novikov Uniformization Theorem, the resulting measurable space is standard Borel, and hence $K^{\otimes V}$ is a fibered probability space over $X$.

\begin{definition}[Bernoulli actions]
Let $\mc{G}$ be a discrete p.m.p.\ groupoid. Let $a: \mc{G}\cc V$ be a discrete action of $\mc{G}$ and let $K$ be a standard probability space. The {\bf generalized Bernoulli action associated to $a$ with base $K$} is the p.m.p.\ action $\beta ^a_{K}: \mc{G} \cc K^{\otimes V}$ defined by taking, for each $g\in \mc{G}$, the transformation $\beta ^a_{K}(g) : K^{V_{s(g)}}\rightarrow K^{V_{r(g)}}$ to be given by
\[
\beta ^a_{K}(g) (f) (v) := f(a(g)^{-1}v )
\]
for $f\in K^{V_{s(g)}}$, $v\in V_{r(g)}$. In the special case where $V=\mc{G}$, with fibering $V_x:=x\mc{G}$ for $x\in \mc{G}^0$, and where $a=\ell _{\mc{G}}$ is the left translation action $\ell _{\mc{G}}: \mc{G}\cc \mc{G}$ given by $\ell _{\mc{G}}(g)(h):=gh$, we write $\beta ^{\mc{G}}_K$ for the p.m.p.\ action $\beta ^{\ell _{\mc{G}}}_K$, and we call the p.m.p.\ action
\[
\beta ^{\mc{G}}_{K}: \mc{G}\cc K^{\otimes \mc{G}}
\]
the {\bf standard Bernoulli action of $\mc{G}$ with base $K$}. If $A$ is a Borel subset of $\mc{G}^0$ then $\mc{G}A$ fibers over $\mc{G}^0$ via $(\mc{G}A)_x:=x\mc{G}A$ for $x\in \mc{G}^0$, and we let $\ell _{\mc{G}A}:\mc{G}\cc \mc{G}A$ denote the left translation action of $\mc{G}$ on $\mc{G}A$, and write $\beta ^{\mc{G}A}_K$ for the action $\beta ^{\ell _{\mc{G}A}}_K$. If no base space $K$ is specified then we will always take the base space to be the unit interval $[0,1]$, equipped with Lebesgue measure $\mu _{[0,1]}$. For example, we will write $\beta ^{\mc{G}A}$ for $\beta ^{\mc{G}A}_{[0,1]}$.
\end{definition}

It is clear that $\beta ^a _K$ and $\beta ^a_M$ are isomorphic whenever the probability spaces $K$ and $M$ are isomorphic, and that $\beta ^a_K$ and $\beta ^b _K$ are isomorphic whenever the discrete actions $a$ and $b$ are isomorphic. We also have the following isomorphisms, whose proof is straightforward.

\begin{proposition}\label{prop:clear}
Let $a:\mc{G}\cc V$ be a discrete action of the discrete p.m.p.\ groupoid $\mc{G}$.
\begin{enumerate}
\item Let $\{ V_i\} _{i\in I}$ be a countable partition of $V$ into $a$-invariant Borel sets, and for each $i\in I$ let $a_i:\mc{G}\cc V_i$ denote the action of $\mc{G}$ restricted to $V_i$. Then $\beta ^a_K \cong \bigotimes _{i\in I}\beta ^{a_i}_K$ for any standard probability space $K$.
\item Let $\{ K_i\} _{i\in I}$ be a countable collection of standard probability spaces and let $K=\prod \{ K_i \, : \, i\in I \}$ be the product space. Then $\bigotimes _{i\in I}\beta ^a _{K_i}\cong \beta ^a _{K}$.
\end{enumerate}
\end{proposition}

\begin{lemma}\label{lem:Biso}
Let $\mc{G}$ be a discrete p.m.p.\ groupoid and let $A$ be a measurable subsets of $\mc{G}^0$ which is a complete unit section for $\mc{G}$. Then the actions $\beta ^{\mc{G}}$ and $\beta ^{\mc{G}A}$ are isomorphic.
\end{lemma}

\begin{proof}
Assume first that $\mc{G}$ is ergodic. Then the hypothesis that the measurable set $A$ is a complete unit sections for $\mc{G}$ is equivalent, modulo a null set, to assuming that $A$ has positive measure.

Case 1: $\mu _{\mc{G}}(A) = 1/n$ for some positive integer $n$. In this case we may find a measurable partition $A_0,A_1,\dots , A_{n-1}$ of $\mc{G}^0$ with $\mu _{\mc{G}}(A_i)=\mu _{\mc{G}}(A)$ for all $i<n$. The set $\mc{G}$ decomposes $\ell _{\mc{G}}$-invariantly as the disjoint union of the sets $\mc{G}A_i$, $0\leq i<n$. Since $\mc{G}$ is ergodic and $\mu _{\mc{G}}(A_i)=\mu _{\mc{G}}(A)$, each of the actions $\ell _{\mc{G}A_i}$ is isomorphic to $\ell _{\mc{G}A}$. This is because, by ergodicity of $\mc{G}$, we can find a measurable bisection $\sigma _i$ with $s(\sigma _i)=A$ and $r(\sigma _i)= A_i$, and hence an isomorphism from $\ell _{\mc{G}A_i}$ to $\ell _{\mc{G}A}$ is given by the map $\mc{G}A_i\rightarrow \mc{G}A$, $g\mapsto g\sigma$. These isomorphisms of discrete actions yield isomorphisms $\beta ^{\mc{G}A_i}_{[0,1]}\cong \beta ^{\mc{G}A}_{[0,1]}$, of p.m.p.\ actions, for all $0\leq i < n$. We therefore have the following isomorphisms of p.m.p.\ actions
\[
\textstyle{\beta ^{\mc{G}}_{[0,1]} \cong \bigotimes _{i<n} \beta ^{\mc{G}A_i}_{[0,1]} \cong \bigotimes _{i<n}\beta ^{\mc{G}A}_{{[0,1]}}\cong \beta ^{\mc{G}A}_{[0,1]^n}  \cong \beta ^{\mc{G}A}_{[0,1]}} ,
\]
where the first and third isomorphisms follow from (1) and (2) respectively of Proposition \ref{prop:clear}, and the last isomorphism holds since the probability spaces $[0,1]^n$ and $[0,1]$ are isomorphic.

Case 2: The general ergodic case. We may find some finite or countably infinite partition $(A_i)_{i\in I}$ of $A$ into measurable sets, such that for each $i\in I$ we have $\mu _{\mc{G}}(A_i)=1/n_i$ for some integer $n_i\geq 1$. By Case 1, we have $\beta ^{\mc{G}A_i}_{[0,1]}\cong \beta ^{\mc{G}}_{[0,1]}$ for each $i\in I$ and hence
\[
\textstyle{\beta ^{\mc{G}A}_{[0,1]}\cong \bigotimes _{i\in I} \beta ^{\mc{G}A_i}_{[0,1]} \cong \bigotimes _{i\in I} \beta ^{\mc{G}}_{[0,1]} \cong \beta ^{\mc{G}}_{[0,1]^{I}} \cong \beta ^{\mc{G}}_{{[0,1]}}},
\]
where the first and third isomorphisms once again follow from (1) and (2) respectively of Proposition \ref{prop:clear}, and the last isomorphism holds since the probability spaces $[0,1]^{I}$ and $[0,1]$ are isomorphic.

Case 3: The general case. Let $\pi : \mc{G}\rightarrow W$ be an ergodic decomposition map for $\mc{G}$, i.e., $\pi$ is a $\mc{G}$-invariant ($\pi (g)= \pi (s(g))=\pi (r(g))$ for all $g\in \mc{G}$) Borel map to a standard Borel space $W$, and for each $w\in W$ the fiber $\mc{G}_w$, equipped with the groupoid operations inherited from $\mc{G}$, is an ergodic discrete p.m.p.\ groupoid with respect to the probability measure $\mu _{\mc{G}^0_w}$ coming from the disintegration of $\mu _{\mc{G}^0}$ over $W$ via the restriction of $\pi$ to $\mc{G}^0$. For each $w\in W$ the Borel set $A_w:= A\cap \mc{G}_w$ is then a complete unit section for $\mc{G}_w$. We view $[0,1]^{\otimes \mc{G}A}$ as a fibered probability space over $W$ via the composition $[0,1]^{\otimes \mc{G}A}\rightarrow \mc{G}^0\rightarrow W$, and likewise for $[0,1]^{\otimes \mc{G}}$. For each $w\in W$, by restricting the p.m.p.\ actions $\beta ^{\mc{G}A}$ and $\beta ^{\mc{G}}$ to the fiber over $w$, we obtain p.m.p.\ actions $(\beta ^{\mc{G}A})_w: \mc{G}_w \cc ([0,1]^{\otimes \mc{G}A})_w$ and $(\beta ^{\mc{G}})_w: \mc{G}_w\cc ([0,1]^{\otimes \mc{G}})_w$ of $\mc{G}_w$ which are seen to coincide with the actions $\beta ^{\mc{G}_wA_w}$ and $\beta ^{\mc{G}_w}$ respectively.

Thus, by Case 2, for each $w\in W$ we can find an isomorphism $T _w : ([0,1]^{\otimes \mc{G}A})_w \rightarrow ([0,1]^{\otimes \mc{G}})_w$ of $\mc{G}_w$-actions $(\beta ^{\mc{G}A})_w$ and $(\beta ^{\mc{G}})_w$. By a standard measurable selection argument, we can choose the assignment $w\mapsto T_w$ in such a way that, after discarding a null set, the disjoint union $T:= \bigsqcup _W T_w$ is measurable, and hence gives an isomorphism of the p.m.p.\ $\mc{G}$-actions $\beta ^{\mc{G}A}$ and $\beta ^{\mc{G}}$. Alternatively, an inspection of the proof of the ergodic case shows that one may choose the isomorphisms $T_w$, $w\in W$, systematically to ensure that their union is measurable.
\end{proof}

\subsection{Groupoid extensions and translation groupoids}\label{sec:ext}
\label{sec:transl}
\begin{definition}[Groupoid extensions]\label{def:ext}
We call a measure preserving groupoid homomorphism $\mc{G}\xra{p}\mc{H}$ between discrete p.m.p.\ groupoids $\mc{G}$ and $\mc{H}$ a {\bf groupoid extension}. Equivalently, a groupoid extension from $\mc{G}$ to $\mc{H}$ is a Borel homomorphism $p$ from $\mc{G}$ to $\mc{H}$ such that (i) $p$ takes $\mu _{\mc{G}^0}$ to $\mu _{\mc{H}^0}$ and (ii) $p$ is locally bijective, i.e., $p$ maps $x\mc{G}$ bijectively onto $p(x)\mc{H}$ for $\mu _{\mc{G} ^0}$-a.e.\ $x\in \mc{G}^0$. We identify two extensions $\mc{G}\xra{p}\mc{H}$ and $\mc{G}\xra{q}\mc{H}$ if they agree on a $\mu _{\mc{G}}$-conull set.
\end{definition}

If $\mc{G}\xra{p}\mc{H}$ is a groupoid extension then we write $p^0$ for the restriction of $p$ to $\mc{G}^0$. By disintegrating $\mu _{\mc{G}}$ via $p$, we will also view $\mc{G}$ as a fibered probability space over $\mc{H}$. Likewise, we will also view $\mc{G}^0$ as a fibered probability space over $\mc{H}^0$.

\begin{definition}[Translation groupoids] Let $\alpha :\mc{H}\cc Z$ be a Borel action of a discrete Borel groupoid $\mc{H}$. The associated {\bf translation groupoid}, denoted $\mc{H}\ltimes Z$, is the discrete Borel \ groupoid on the set $\mc{H}\ltimes Z := \mc{H}\otimes _{\mc{H}^0}Z$ with unit space $(\mc{H}\ltimes Z )^0 := \mc{H}^0 \otimes _{\mc{H}^0}Z$, source and range maps
\[
s(h,z):= (s_{\mc{H}}(h),z), \text{ and } r(h,z):= (r_\mc{H}(h), \alpha (h)z ) ,
\]
and multiplication and inversion defined by $(h_1, \alpha (h_0)z )(h_0, z) := (h_1h_0, z )$, and $(h,z)^{-1}:=(h^{-1},\alpha (h)z )$. This makes $\mc{H}\ltimes Z$ a discrete Borel groupoid, whose unit space $(\mc{H}\ltimes Z )^0$, viewed as a fibered Borel space over $\mc{H}^0$, is naturally isomorphic to $Z$ via the right projection map $(\mc{H}\ltimes Z)^0\rightarrow Z$. The left projection map $\mc{H}\ltimes Z \rightarrow \mc{H}$ is then a locally bijective Borel groupoid homomorphism

If $\mc{H}$ is a discrete p.m.p.\ groupoid and $\alpha : \mc{H}\cc Z$ is a p.m.p.\ action, then we naturally view the translation groupoid $\mc{H}\ltimes Z$ as a discrete p.m.p.\ groupoid by equipping it with the measure $\mu _{\mc{H}\ltimes Z} := \mu _{\mc{H}} \otimes _{\mc{H}^0} \mu _Z$, so that $\mu _{(\mc{H}\ltimes Z )^0} = \mu _{\mc{H}^0}\otimes _{\mc{H}^0} \mu _Z$. This makes the left projection map $\mc{H}\ltimes Z \rightarrow \mc{H}$ an extension of discrete p.m.p.\ groupoids.
\end{definition}

\begin{remark}[Category theoretic remark]\label{rem:category} We obtain a category $\bm{\mathrm{DPG}}$, whose objects are discrete p.m.p.\ groupoids, and whose morphisms are given by (equality-a.e.\ equivalence classes of) groupoid extensions. We let $\bm{\mathrm{DPG}^2}$ denote the associated arrow category, i.e., whose objects are groupoid extensions and whose morphisms are commuting squares, i.e., a morphism from $\mc{G}_0\xra{p_0} \mc{H}_0$ to $\mc{G}_1\xra{p_1}\mc{H}_1$ is a pair $(\mc{G}_0\xra{p_{\mc{G}}} \mc{G}_1, \mc{H}_0\xra{p_{\mc{H}}} \mc{H}_1 )$ of extensions with $p_1\circ p_{\mc{G}}=p_{\mc{H}}\circ p_0$. We will also be interested in the slice category $\bm{\mathrm{DPG}}_{\mc{H}}$ over a fixed discrete p.m.p.\ groupoid $\mc{H}$, whose objects are extensions $\mc{G}\rightarrow \mc{H}$ of $\mc{H}$, and whose morphisms from $\mc{G}_0\xra{p_0}\mc{H}$ to $\mc{G}_1\xra{p_1}\mc{H}$ are given by groupoid extensions $\mc{G}_0\xra{p}\mc{G}_1$ satisfying $p_1\circ p = p_0$. When we wish to indicate this context, we will call an object of $\bm{\mathrm{DPG}}_{\mc{H}}$ an {\bf $\mc{H}$-extension}, and a morphism of $\bm{\mathrm{DPG}}_{\mc{H}}$ a {\bf morphism of $\mc{H}$-extensions}.

There is a straightforward equivalence between the category $\bm{\mathrm{DPG}}_{\mc{H}}$, of $\mc{H}$-extensions, and the category of p.m.p.\ actions of $\mc{H}$ (with $\mc{H}$-equivariant fiberwise measure preserving maps as morphisms) which, at the level of objects, is implemented by the translation groupoid construction together with the following proposition, whose proof may be found in \cite[Lemma 3.27]{GL17b}.
\end{remark}

\begin{proposition}\label{prop:ExtActCorr}
Let $\mc{G}$ and $\mc{H}$ be discrete Borel groupoids and let $p:\mc{G}\rightarrow \mc{H}$ be a locally bijective Borel groupoid homomorphism. Then there exists a unique Borel action $\alpha _p : \mc{H}\cc \mc{G}^0$, of $\mc{H}$ on the fibered Borel space $p^0:\mc{G}^0\rightarrow \mc{H}^0$, satisfying $\alpha _p(p(g))s_{\mc{G}}(g) = r_{\mc{G}}(g)$ for all $g\in \mc{G}$.

If $\mc{G}$ and $\mc{H}$ are discrete p.m.p.\ groupoid and $\mc{G}\xra{p}\mc{H}$ is an $\mc{H}$-extension, then the action $\alpha _p : \mc{H}\cc \mc{G}^0$ is a p.m.p.\ action on the fibered probability space $p^0: \mc{G}^0\rightarrow \mc{H}^0$, and the map $p\otimes s_{\mc{G}}$: $g\mapsto (p(g),s_{\mc{G}}(g))$ is an isomorphism of $\mc{H}$-extensions:
\[
\xymatrix{
\mc{G}\ar[d]_{p} \ar[rr]^{p\otimes s_{\mc{G}}}_{\cong} & \    &\mc{H}\ltimes \mc{G}^0 \ar[d] \\
\mc{H}\ar[rr]^{\mathrm{id}} & \     &\mc{H}
}
\]
\end{proposition}

\subsection{Products and relatively independent products of extensions} Let $\mc{G}$, $\mc{H}$, and $\mc{K}$ be discrete Borel groupoids and let $p : \mc{G}\rightarrow \mc{H}$ and $q:\mc{K}\rightarrow \mc{H}$ be Borel measurable groupoid homomorphisms. The fibered product $\mc{G}\otimes _{\mc{H}}\mc{K}$ is naturally endowed with the structure of a discrete Borel groupoid, by taking the unit space to be the fibered product $(\mc{G}\otimes _{\mc{H}}\mc{K} )^0 := \mc{G}^0\otimes _{\mc{H}^0}\mc{K}^0$, and performing all groupoid operations coordinate-wise, i.e., defining the source, range, multiplication, and inversion maps respectively by $s(g,k):= (s_{\mc{G}}(g),s_{\mc{K}}(k))$ and $r(g,k):= (r_{\mc{G}}(g),r_{\mc{K}}(k))$, $(g_1,k_1)(g_0,k_0):= (g_1g_0,k_1k_0)$, and $(g,k)^{-1}:= (g^{-1},k^{-1})$.

\begin{definition}[Relatively Independent products]
Let $\mc{H}$ be a discrete p.m.p.\ groupoid and let $\mc{G}\xra{p}\mc{H}$ and $\mc{K}\xra{q} \mc{H}$ be two $\mc{H}$-extensions. The relatively independent product measure $\mu _{\mc{G}\otimes _{\mc{H}}\mc{K}}$ then makes the fibered product groupoid $\mc{G}\otimes _{\mc{H}}\mc{K}$ a discrete p.m.p.\ groupoid, called the {\bf relatively independent product} over $\mc{H}$ of $\mc{G}$ and $\mc{K}$. The unit space of $\mc{G}\otimes _{\mc{H}}\mc{K}$, as a probability space, is seen to be the relatively independent product $(\mc{G}\otimes _{\mc{H}}\mc{K} )^0 = \mc{G}^0\otimes _{\mc{H}^0}\mc{K}^0$ over $\mc{H}^0$ of the fibered probability spaces $p^0:\mc{G}^0\rightarrow \mc{H}^0$ and $q^0:\mc{K}^0\rightarrow \mc{H}^0$. The left and right projection maps, $\mc{G}\otimes _{\mc{H}}\mc{K} \xra{p_{\mc{G}}}\mc{G}$ and $\mc{G}\otimes _{\mc{H}}\mc{K} \xra{p_{\mc{K}}} \mc{K}$, are then groupoid extensions, and we have the $\mc{H}$-extension $\mc{G}\otimes _{\mc{H}}\mc{K}\rightarrow \mc{H}$ via the map $p\circ p_{\mc{G}} = q\circ p_{\mc{K}}$.
\end{definition}

\begin{remark}\label{rem:translext} If $\mc{H}\cc Y$ and $\mc{H}\cc Z$ are p.m.p.\ actions of $\mc{H}$ then the map $((h,y),(h,z))\mapsto (h,(y,z))$ gives an isomorphism from the relatively independent product $(\mc{H}\ltimes Y)\otimes _{\mc{H}} (\mc{H}\ltimes Z )$ of the associated translation groupoids, to the translation groupoid $\mc{H}\ltimes (Y\otimes _{\mc{H}^0}Z )$ associated to the product action $\mc{H}\cc Y\otimes _{\mc{H}^0}Z$.
\end{remark}

\begin{remark}\label{rem:prodprod}
Parallel to Remark \ref{rem:translext}, if $\mc{K}$, $\mc{G}$, and $\mc{L}$ are all extensions of $\mc{H}$, then the $\mc{K}$-extensions $(\mc{K}\otimes _{\mc{H}}\mc{G})\otimes _{\mc{K}}(\mc{K}\otimes _{\mc{H}}\mc{L})\rightarrow \mc{K}$ and $\mc{K}\otimes _{\mc{H}}(\mc{G}\otimes _{\mc{H}}\mc{L})\rightarrow \mc{K}$ are isomorphic via $((k,g),(k,l))\mapsto (k,(g,l))$.
\end{remark}

\begin{definition}[Independent products]
The {\bf independent product} of two discrete p.m.p.\ groupoids $\mc{G}$ and $\mc{K}$ is the discrete p.m.p.\ groupoid $\mc{G}\otimes \mc{K}$ with unit space $\mc{G}^0\otimes \mc{K}^0$, and all groupoid operations performed coordinate-wise.
\end{definition}

\begin{definition}[Ergodic and weakly mixing extensions]
A groupoid extension $\mc{G}\xra{p}\mc{H}$ is said to be {\bf relatively ergodic} if for every measurable subset $A$ of $\mc{G}^0$ which is $\mc{G}$-invariant (i.e., with $s(g)\in A$ if and only if $r(g)\in A$ for $\mu _{\mc{G}}$-a.e.\ $g\in \mc{G}$) there is some measurable subset $B$ of $\mc{H}^0$ such that $\mu _{\mc{G}^0}(A\triangle p^{-1}(B))=0$. 
A groupoid extension $\mc{G}\ra \mc{H}$ is said to be {\bf relatively weakly mixing} if the extension $\mc{G}\otimes _{\mc{H}}\mc{G}\rightarrow \mc{H}$ is relatively ergodic.
\end{definition}

We will use the following well-known characterization of relatively weakly mixing extensions.

\begin{lemma}\label{lem:Gweakmix}
Let $\mc{G}\rightarrow \mc{H}$ be an extension of discrete p.m.p.\ groupoids. Then the following are equivalent:
\begin{enumerate}
\item The extension $\mc{G}\rightarrow \mc{H}$ is relatively weakly mixing.
\item The extension $\mc{G}\otimes _{\mc{H}}\mc{G} \rightarrow \mc{H}$ is relatively weakly mixing.
\item Given any extension $\mc{K}\rightarrow \mc{H}$, the extension $\mc{K}\otimes _{\mc{H}}\mc{G}\rightarrow \mc{K}$ is relatively ergodic.
\end{enumerate}
\end{lemma}

\subsection{Equivalence of homomorphisms} Let $\mc{G}$ be a discrete p.m.p.\ groupoid, and let $L$ be a Polish group. Two measurable homomorphisms $u:\mc{G}\rightarrow L$ and $v:\mc{G}\rightarrow L$ into $L$ are said to be {\bf equivalent} if there exists a measurable map $F: \mc{G}^0 \rightarrow L$ such that $F(r(g))u(g)F(s(g))^{-1} = v(g)$ for a.e.\ $g\in \mc{G}$. This is an equivalence relation on measurable homomorphisms from $\mc{G}$ to $L$.

We let $Z(\mc{G},L)$ denote the set of all measurable homomorphisms from $\mc{G}$ to $L$, and we let $H(\mc{G},L )$ denote the set of all equivalence classes of measurable homomorphisms from $\mc{G}$ to $L$.

\subsection{Reductions}\label{sec:red} Let $\mc{H}$ be a discrete p.m.p.\ groupoid and let $A$ be a non-null measurable subset of $\mc{H}^0$. Then we endow the reduction, $\mc{H}_A=A\mc{H}A$, of $\mc{H}$ to $A$, with the structure of a discrete p.m.p.\ groupoid by taking $\mu _{\mc{H}_A^0}$ to be the normalized restriction of $\mu _{\mc{H}^0}$ to $A$. 
If $w:\mc{H}\rightarrow L$ is a homomorphism to a group $L$ then we let $w_A$ denote the restriction of $w$ to $\mc{H}_A$.

Let $\mc{G}\xra{p}\mc{H}$ be an $\mc{H}$-extension. Then we abuse notation and write $\mc{G}_A$ for the reduction $\mc{G}_{p^{-1}(A)}$ of $\mc{G}$ to $p^{-1}(A)\subseteq \mc{G}^0$. We let $\mc{G}_A\xra{p_A} \mc{H}_A$ denote the $\mc{H}_A$-extension obtained by restricting $p$. Likewise, if $w:\mc{G}\rightarrow L$ is a homomorphism to a group $L$ then we write $w_A$ for the restriction of $w$ to $\mc{G}_A$, and if $\mc{G}_0\xra{q}\mc{G}_1$ is a morphism of $\mc{H}$-extensions from $\mc{G}_0\xra{p_0}\mc{H}$ to $\mc{G}_1\xra{p_1}\mc{H}$ then we let $(\mc{G}_0)_A\xra{q_A}(\mc{G}_1)_A$ denote the corresponding morphism of $\mc{H}_A$-extensions, from $(p_0)_A$ to $(p_1)_A$.

\begin{proposition}\label{prop:restrict}
Let $\mc{H}$ be a discrete p.m.p.\ groupoid, let $A\subseteq \mc{H}^0$ be a measurable complete unit section for $\mc{H}$, and let $L$ be a Polish group.
\begin{enumerate}
\item The restriction map $Z(\mc{H}, L ) \rightarrow Z(\mc{H}_A, L)$, $w\mapsto w_A$, is surjective, and descends to a bijection $H(\mc{H}, L ) \rightarrow H(\mc{H}_A , L)$ on the space of equivalence classes of homomorphisms.
\item Let $\mc{G}\xra{p}\mc{H}$ be a groupoid extension and let $w \in Z(\mc{G},L )$. Then $w$ is equivalent to a homomorphism which descends to $\mc{H}$ if and only if $w_A$ is equivalent to a homomorphism which descends to $\mc{H}_A$.
\end{enumerate}
\end{proposition}

\begin{proof}
Since $\mc{H}A\mc{H}=\mc{H}$, the countable-to-one Borel map $s: A\mc{H}\rightarrow \mc{H}^0$ is surjective, and it is clearly injective on $A$. By the Lusin-Novikov Uniformization Theorem we may therefore find a Borel subset $\psi$ of $A\mc{H}$, containing $A$, with $|\psi \cap s^{-1}(x) | =1$ for all $x\in \mc{H}^0$. Then for each $h\in \mc{H}$ the sets $\psi h$, $h\psi ^{-1}$, and $\psi h\psi ^{-1}$ each consist of a single element of $A\mc{H}$, $\mc{H}A$, and $\mc{H}_A$ respectively, and by abuse of notation we will also denote these elements by $\psi h$, $h\psi ^{-1}$, and $\psi h\psi ^{-1}$ respectively. The map $c: \mc{H}\rightarrow \mc{H}_A$, $c(h):=\psi h \psi ^{-1}$, is then a groupoid homomorphism extending the identity map on $\mc{H}_A$. The homomorphism $c$ is not measure preserving in general, but it is measure-class preserving since, for a subset $D$ of $\mc{H}_A$ we have $c^{-1}(D)\subseteq \mc{H}D\mc{H}$, so that $(c_*\mu _{\mc{H}})(D)=0$ if and only if $\mu _{\mc{H}}(\mc{H}D\mc{H})=0$ if and only if $\mu _{\mc{H}}(D) =0$ if and only if $\mu _{\mc{H}_A}(D)=0$.

(1): It is clear that the restriction map descends to a well-defined map $H(\mc{H},L)\rightarrow H(\mc{H}_A,L)$. The map $w\mapsto w_A$ is surjective, since if $v\in Z(\mc{H}_A, L)$ then the homomorphism $w\in Z(\mc{H}, L )$, defined by $w(h) := v(\psi h\psi ^{-1})$, satisfies $w_A=v$. To see that the map on equivalence classes is injective, let $w^0, w^1 \in Z(\mc{H}, L )$ be such that $w^0_A$ is equivalent to $w^1_A$, so that there exists a measurable map $f: A \rightarrow L$ with $w^0(h)=f(r(h))w^1(h)f(s(h))^{-1}$ for a.e.\ $h\in \mc{H}_A$. Then for a.e.\ $h\in \mc{H}$ we have
\begin{align}
\nonumber w^0 (\psi h \psi ^{-1}) &= f(r(\psi h \psi ^{-1}))w^1(\psi h \psi ^{-1})f(s(\psi h \psi ^{-1}))^{-1} \\
\label{eqn:php}  &= f(r(\psi r(h)))w^1(\psi h \psi ^{-1})f(r(\psi s(h)))^{-1} ,
\end{align}
where the identities $r(\psi h \psi ^{-1})=r(\psi r(h))$ and $s(\psi h \psi ^{-1})=r(\psi s(h))$ follow from $\psi h \psi ^{-1}= (\psi r(h))h(\psi s(h))^{-1}$. Define $F:\mc{H}^0 \rightarrow L$ by
\[
F(x) := w^0(\psi x )^{-1}f(r(\psi x)) w^1(\psi x )
\]
for $x\in \mc{H}^0$. Then for a.e.\ $h\in \mc{H}$ we have
\begin{align*}
w^0(h)&= w ^0(\psi r(h))^{-1} w^0(\psi h \psi ^{-1}) w^0(\psi s(h)) \\
&= w ^0(\psi r(h))^{-1} f(r(\psi r(h)))w^1(\psi h \psi ^{-1}) f(r(\psi s(h)))^{-1}w^0(\psi s(h)) \\
&= w ^0(\psi r(h))^{-1} f(r(\psi r(h)))w^1(\psi r(h))w^1(h)w^1(\psi s(h))^{-1}f(r(\psi s(h)))^{-1}w^0(\psi s(h)) \\
&=F(r(h))w^1(h)F(s(h))^{-1},
\end{align*}
where the second equality follows from \eqref{eqn:php}, and the first and third equalities come from applying the homomorphisms $w^0$ and $w^1$ respectively to the identity $\psi h \psi ^{-1} = (\psi r(h))h (\psi s(h))^{-1}$. This shows that $w^0$ and $w^1$ are equivalent.

(2): If $w$ is equivalent to a homomorphism $u$ which descends to $\mc{H}$, then $w_A$ is equivalent to the homomorphism $u_A$, and $u_A$ descends to $\mc{H}_A$. Conversely, assume that $w_A$ is equivalent to a homomorphism which descends to $\mc{H}_A$, i.e., $w_A$ is equivalent to a homomorphism of the form $v\circ p_A$, where $v\in Z(\mc{H}_A,L)$. By part (1) there is some $u\in Z(\mc{H},L)$ such that $u_A=v$. Then $(u\circ p)_A=v\circ p_A$ is equivalent to $w_A$, so $u\circ p$ is equivalent to $w$ by part (1).
\end{proof}

\begin{proposition}\label{prop:extact}
Let $\mc{H}$ be a discrete p.m.p.\ groupoid and let $A\subseteq \mc{H}^0$ be a measurable complete unit section for $\mc{H}$. Let $\mc{K}\xra{q}\mc{H}_A$ be an $\mc{H}_A$-extension. Then there exists an $\mc{H}$-extension $\mc{G}\xra{p}\mc{H}$ whose reduction $\mc{G}_A\xra{p_A}\mc{H}_A$ is isomorphic as an $\mc{H}_A$-extension to $\mc{K}\xra{q}\mc{H}_A$.

Moreover, this $\mc{H}$-extension is unique: if $\mc{G}_0\xra{p_0}\mc{H}$ and $\mc{G}_1\xra{p_1}\mc{H}$ are two $\mc{H}$-extensions, and if $(\mc{G}_0)_A\xra{t}(\mc{G}_1)_A$ is an isomorphism of the $\mc{H}_A$-extensions $(p_0)_A$ and $(p_1)_A$, then there exists a unique isomorphism of $\mc{H}$-extensions $\mc{G}_0\xra{T}\mc{G}_1$ with $T_A = t$.
\end{proposition}

\begin{proof}
As in the proof of Proposition \ref{prop:restrict}, we fix a Borel subset $\psi$ of $\mc{H}$ with $A\subseteq \psi \subseteq A\mc{H}$ and $|\psi \cap s^{-1}(x) | =1$ for all $x\in \mc{H}^0$, and for each $h\in \mc{H}$ we identify $\psi h$, $\psi h\psi ^{-1}$, and $h\psi ^{-1}$ with elements of $\mc{H}$, so that the map $c: \mc{H}\rightarrow \mc{H}_A$, $c(h):= \psi h \psi ^{-1}$, is a measure-class preserving groupoid homomorphism which restricts to the identity map on $\mc{H}_A$.

As a discrete Borel groupoid we take $\mc{G}$ to be the fibered product $\mc{G}:=\mc{H}\otimes _{\mc{H}_A}\mc{K}$ with respect to the homomorphisms $c$ and $q$, and we take $p$ to be the left projection from $\mc{G}$ to $\mc{H}$. We equip the unit space $\mc{G}^0 = \mc{H}^0\otimes _{\mc{H}_A^0}\mc{K}^0$ with the probability measure
\begin{equation*}
\mu _{\mc{G}^0} := \int _{\mc{H}^0}\updelta _x \times \mu _{\mc{K}^0_{c(x)}} \, d\mu _{\mc{H}^0}(x).
\end{equation*}
These definitions make $\mc{G}$ a discrete p.m.p.\ groupoid, $\mc{G}\xra{p}\mc{H}$ an $\mc{H}$-extension, and the right projection $\mc{G}_A\ra\mc{K}$ provides an isomorphism of $\mc{H}_A$-extensions from $\mc{G}_A\xra{p_A}\mc{H}_A$ to $\mc{K}\xra{q}\mc{H}_A$.

These claims can be verified using translation groupoids as follows. The p.m.p.\ action $\alpha : \mc{H}_A\cc \mc{K}^0$, obtained from the extension $q$ by Proposition \ref{prop:ExtActCorr}, gives rise to a p.m.p.\ action $\beta : \mc{H}\cc \mc{G}^0$ defined by $\beta (h)(s(h),y):= (r(h), \alpha (c(h))y)$ for $h\in \mc{H}$, $y\in \mc{K}^0_{c(s(h))}$. By Proposition \ref{prop:ExtActCorr}, the $\mc{H}_A$-extension $\mc{K}\xra{q}\mc{H}_A$ is isomorphic to the $\mc{H}_A$-extension $\mc{H}_A\ltimes \mc{K}^0\rightarrow \mc{H}_A$ associated to $\alpha$, and this in turn is isomorphic as an $\mc{H}_A$-extension to the reduction $(\mc{H}\ltimes \mc{G}^0)_A\rightarrow \mc{H}_A$, of the $\mc{H}$-extension associated to the translation groupoid of $\beta$, via the map $(h,y)\mapsto (h,(s(h),y))$. The Borel map $p\otimes s_{\mc{G}}:\mc{G}\rightarrow \mc{H}\ltimes \mc{G}^0$, $(h,k)\mapsto (h,(s(h),s(k)))$, is then a bijective groupoid homomomorphism which sends $\mu _{\mc{G}^0}$ to $\mu _{(\mc{H}\ltimes \mc{G}^0)^0}$. The claims from the previous paragraph now follow.

We now prove the uniqueness statement. By Proposition \ref{prop:ExtActCorr} it is enough to consider the case where each $\mc{G}_i$ is a translation groupoid $\mc{G}_i = \mc{H}\ltimes Z_i$ associated to a p.m.p.\ action $\alpha _i:\mc{H}\cc Z_i$ of $\mc{H}$, and $\mc{G}_i\xra{p_i}\mc{H}$ is the projection map $(h,z)\mapsto h$. Let $(\mc{G}_0)_A \xra{t}(\mc{G}_1)_A$ be an isomorphism of $\mc{H}_A$-extensions. Then $t$ has the form $t(h,z)=(h,t^0(z))$, where $t^0 : (Z_0)_A \rightarrow (Z_1)_A$ is an isomorphism of the reduced actions $\alpha _0 |_{\mc{H}_A}:\mc{H}_A \cc (Z_0)_A$ and $\alpha _1|_{\mc{H}_A}:\mc{H}_A\cc (Z_1)_A$. Define $T:\mc{G}_0\rightarrow \mc{G}_1$ by $T(h,z)= (h,T^0(z))$, where $T^0:Z_0\rightarrow Z_1$ is defined fiberwise by taking, for $x\in \mc{H}^0$, the restriction $T^0_x:(Z_0)_x\rightarrow (Z_1)_x$ to $(Z_0)_x$ to be given by
\[
T^0_x =\alpha _1(\psi x)^{-1}\circ t^0_{r(\psi x)}\circ \alpha _0(\psi x).
\]
Then $T^0$ is an isomorphism of the actions $\alpha _0$ and $\alpha _1$ which extends the isomorphism $t^0$, hence $T$ is an isomorphism of $\mc{H}$-extensions with $T_A=t$. To see that $T$ is unique, suppose that $U : \mc{G}_0\rightarrow \mc{G}_1$ is another isomorphism of $\mc{H}$-extensions with $U_A=t$, so that $U(h,z)=(h,U^0(z))$ for some isomorphism $U^0: Z_0\rightarrow Z_1$ extending $t^0$. Then for any $x\in \mc{H}^0$, the restriction $U_x$, of $U$ to $(Z_0)_x$, satisfies $U^0_x = \alpha _1(\psi x)^{-1}\circ U^0_{r(\psi x)} \circ \alpha _0(\psi x)=\alpha _1(\psi x)^{-1}\circ t^0_{r(\psi x)}\circ \alpha _0(\psi x) = T^0_x$, whence $U=T$.
\end{proof}

\begin{proposition}\label{prop:ExtIsoRed}
Let $\mc{G}$ and $\mc{H}$ be ergodic discrete p.m.p.\ groupoids, let $A\subseteq \mc{G}^0$ and $B\subseteq \mc{H}^0$ be positive measure sets with $\mu _{\mc{G}}(A) = \mu _{\mc{H}}(B)$, and suppose that $t:\mc{G}_A\rightarrow \mc{H}_B$ is an isomorphism. Then $t$ extends to an isomorphism $T:\mc{G}\rightarrow \mc{H}$.
\end{proposition}

\begin{proof}
Let $A_0,A_1,\dots$ be a sequence of non-null measurable subsets of $A$, with $A_0 = A$, and $\sum _{i\in \N }\mu _{\mc{G}}(A_i) =1$ (we allow some of the $A_i$'s to be empty). Then for each $i\in \N$, letting $B_i:=t(A_i)$, we have $\mu _{\mc{H}}(B_i)= \mu _{\mc{G}}(A_i)$ since $t$ is measure preserving and $\mu _{\mc{G}}(A)=\mu _{\mc{H}}(B)$. By ergodicity, after discarding a null set, we may find a sequence $\sigma _0,\sigma _1,\dots$, of measurable bisections of $\mc{G}$ with $s(\sigma _0), s(\sigma _1),\dots$ partitioning $\mc{G}^0$, and with $\sigma _0=A$ and $r(\sigma _i ) = A_i$ for all $i\in \N$. Likewise we may find measurable bisections $\tau _0,\tau _1,\dots$, of $\mc{H}$ with $s(\tau _0), s(\tau _1),\dots$ partitioning $\mc{H}^0$, and with $\tau _0 = B_0=B$ and $r(\tau _i)=B_i$ for all $i\in \N$. For $x\in \mc{G}^0$ define $\sigma _x := \sigma _i$ and $\tau _x := \tau _i$ if and only if $x\in s(\sigma _i)$. Then $\sigma _{r(g)}g\sigma _{s(g)}^{-1}  \in \mc{G}_A$ for all $g\in \mc{G}$, and the map $T:\mc{G}\rightarrow \mc{H}$ defined by
\[
T(g) := \tau _{r(g)}^{-1}t(\sigma _{r(g)}g\sigma _{s(g)}^{-1} )\tau _{s(g)} ,
\]
is easily verified to be an isomorphism extending $t$.
\end{proof}

\begin{proposition}\label{prop:amp}
Let $\mc{H}$ be an ergodic discrete p.m.p.\ groupoid and assume $\mu _{\mc{H}^0}$ has no atoms. Then for each real number $t\geq 1$, up to isomorphism there exists a unique ergodic discrete p.m.p.\ groupoid $\mc{G}$ such that $\mc{H}$ is isomorphic to $\mc{G}_A$ for some measurable $A\subseteq \mc{G}^0$ with $\mu _{\mc{G}}(A) = 1/t$.
\end{proposition}

\begin{proof}
The uniqueness of $\mc{G}$ follows from Proposition \ref{prop:ExtIsoRed}, so it remains prove existence. Let $n$ be any integer greater than $t$. Let $[n]:= \{ 0,1,\dots ,n-1 \}$ equipped with normalized counting measure $\mu _{[n]}$ and let $\mc{J}_n := [n]\times [n]$ be the equivalence relation on $[n]$ consisting of a single equivalence class. Then $\mc{J}_n$ is an ergodic discrete p.m.p.\ groupoid, whose unit space is $\mc{J}_n^0 = \{ (0,0), (1,1),\dots , (n,n) \} $. Let $\mc{K}:= \mc{H}\otimes \mc{J}_n$ and let $A=\mc{H}^0\otimes \{ (0,0) \}$, so that $\mu _{\mc{K}}(A)=1/n$. Let $B$ be any measurable subset of $\mc{K}^0$ containing $A$ with $\mu _{\mc{K}}(B) = t/n$ and let $\mc{G}:=\mc{K}_B$. Since $\mc{H}$ and $\mc{J}_n$ are both ergodic, so is $\mc{K}$, and hence so is $\mc{G}$. We have $\mu _{\mc{G}}(A)= \mu _{\mc{K}}(A)/\mu _{\mc{K}}(B)= 1/t$, and $\mc{H}$ is isomorphic to $\mc{G}_A=\mc{K}_A$ via the map $h\mapsto (h,(0,0))$.
\end{proof}

\begin{proposition}\label{prop:IsoRed}
Let $\mc{G}$ be an ergodic discrete p.m.p.\ groupoid and let $A$ and $B$ be positive measure subsets of $\mc{G}^0$ with $\mu _{\mc{G}}(A) =\mu _{\mc{G}}(B)$. Then the reductions $\mc{G}_A$ and $\mc{G}_B$ are isomorphic.
\end{proposition}

\begin{proof}
Since $\mc{G}$ is ergodic and $\mu _{\mc{G}}(A) =\mu _{\mc{G}}(B)$, we can find a measurable bisection $\sigma$ of $\mc{G}$ with $s(\sigma ) = A$ and $r(\sigma ) = B$. Then the map $\mc{G}_A \rightarrow \mc{G}_B$, $g\mapsto \sigma g \sigma ^{-1}$ is an isomorphism.
\end{proof}

\subsection{Untwisting lemma}
The following is Lemma 3.2 of \cite{Fu07} stated for discrete p.m.p. groupoids (see also Lemma 3.11 of \cite{Pop07}). We note that while the statement of Lemma 3.2 of \cite{Fu07} corresponds to the case where $\mc{G}$ is assumed to be ergodic, this assumption is unnecessary and is easily removed.

\begin{lemma}\label{lem:Gbasic}
Let $\mc{G}\xra{p}\mc{H}$ be a relatively weakly mixing extension of discrete p.m.p.\ groupoids. Let $L$ be a Polish group admitting a bi-invariant metric, and let $u:\mc{H}\rightarrow L$ and $v:\mc{H}\rightarrow L$ be measurable maps. Suppose that $\Phi :\mc{G}^0 \rightarrow L$ is a measurable map satisfying, for a.e.\ $g\in \mc{G}$,
\[
u(p(g)) = \Phi (r(g))v(p(g))\Phi (s(g))^{-1} .
\]
Then $\Phi$ descends to $\mc{H}^0$, i.e., there is some measurable map $f :\mc{H}^0\rightarrow L$ such that $\Phi (x) = f(p (x))$ for a.e.\ $x\in \mc{G}^0$.
\end{lemma}

We will need the following "asymmetric" generalization of \cite[Theorem 3.4]{Fu07} (see also \cite[Theorem 3.1]{Pop07}). We provide a complete proof, although we note that the proof in this asymmetric setting is nearly identical to the proof of \cite[Theorem 3.4]{Fu07}.

\begin{lemma}[Untwisting lemma, asymmetric version]\label{lem:GFurman}
Let $\mc{H}$ be a discrete p.m.p.\ groupoid and let $\mc{G}_0 \xra{q_0} \mc{H}$ and $\mc{G}_1\xra{q_1} \mc{H}$ be two $\mc{H}$-extensions. Let $\mc{G}_0\otimes _{\mc{H}}\mc{G}_1$ be the associated relatively independent product over $\mc{H}$, with projections $p_0$ and $p_1$ to $\mc{G}_0$ and $\mc{G}_1$ respectively:
\[
\xymatrix{
\mc{G}_0\otimes _{\mc{H}}\mc{G}_1 \ar[d]_{p_0}\ar[rr]^{p_1}  &    &\mc{G}_1 \ar[d]^{q_1} \\
\mc{G}_0\ar[rr]^{q_0}    &   & \mc{H}
}
\]
Assume that the extension $\mc{G}_1\xra{q_1}\mc{H}$ is relatively weakly mixing. Let $w_0 : \mc{G}_0 \rightarrow L$ and $w_1 :\mc{G}_1 \rightarrow L$ be measurable homomorphisms to a Polish group $L$ which admits a bi-invariant metric, and suppose that the homomorphisms $w_0\circ p_0$ and $w_1\circ p_1$ are equivalent. Then there exists a measurable homomorphism $w: \mc{H}\rightarrow L$ such that $w_0$ is equivalent to $w\circ q_0$ and $w_1$ is equivalent to $w\circ q_1$.

In fact, if $F:(\mc{G}_0\otimes _{\mc{H}}\mc{G}_1)^0 \rightarrow L$ is a measurable map witnessing that $w_0\circ p_0$ and $w_1\circ p_1$ are equivalent, so that
\begin{equation}\label{eqn:GFw}
w_0(g_0) = F(r(g_0, g_1)) w_1(g_1)F(s(g_0, g_1))^{-1}
\end{equation}
for a.e.\ $(g_0,g_1) \in \mc{G}_0\otimes _{\mc{H}}\mc{G}_1$, then there exist measurable maps $\varphi _0:\mc{G}_0^0\rightarrow L$ and $\varphi _1:\mc{G}_1^0\rightarrow L$, and a measurable homomorphism $w:\mc{H}\rightarrow L$ such that
\begin{align*}
F(x_0,x_1)&=\varphi _0(x_0)^{-1}\varphi _1(x_1) \\
w (p_0(g_0)) &= \varphi _0(r(g_0))w_0(g_0)\varphi _0 (s(g_0))^{-1}\\
w(p_1(g_1)) & = \varphi _1(r(g_1))w_1(g_1)\varphi _1 (s(g_1))^{-1}
\end{align*}
for a.e.\ $(x_0,x_1)\in (\mc{G}_0 \otimes _{\mc{H}}\mc{G}_1)^0$, $g_0\in \mc{G}_0$, and $g_1\in \mc{G}_1$.
\end{lemma}

\begin{proof}
We begin with a claim.
\begin{claim}\label{claim:Gwm} The extension $(\mc{G}_0\otimes _\mc{H}\mc{G}_0)\otimes _{\mc{H}}\mc{G}_1  \rightarrow \mc{G}_0\otimes _\mc{H}\mc{G}_0$ is relatively weakly mixing.
\end{claim}
\begin{proof}[Proof of Claim \ref{claim:Gwm}]
Define $\mc{K}_i := \mc{G}_i\otimes _\mc{H}\mc{G}_i$ for $i=0,1$, and let $\mc{K}:= \mc{K}_0\otimes _{\mc{H}}\mc{G}_1$. We must show that the extension $\mc{K}\otimes _{\mc{K}_0} \mc{K} \rightarrow \mc{K}_0$ is relatively ergodic. This $\mc{K}_0$-extension is isomorphic to the $\mc{K}_0$-extension $\mc{K}_0\otimes _{\mc{H}}\mc{K}_1 \rightarrow \mc{K}_0$ by Remark \ref{rem:prodprod}. Since $\mc{G}_1 \rightarrow  \mc{H}$ is relatively weakly mixing, by applying (1)$\Rightarrow$(2) of Lemma \ref{lem:Gweakmix} we see that the extension $\mc{K}_1 \rightarrow \mc{H}$ is relatively weakly mixing, and hence we conclude from (1)$\Rightarrow$(3) of Lemma \ref{lem:Gweakmix} that the extension $\mc{K}_0 \otimes _{\mc{H}} \mc{K}_1 \rightarrow \mc{K}_0$ is relatively ergodic. \qedhere[Claim \ref{claim:Gwm}]
\end{proof}

Keeping the notation from the claim, let $\Phi :\mc{K}^0 \rightarrow L$ be the map $\Phi ((x_0,x_0'),x_1) := F(x_0,x_1)F(x_0',x_1)^{-1}$, for $((x_0,x_0'),x_1)\in \mc{K}^0$. Then \eqref{eqn:GFw} implies that
\[
\Phi (r((g_0,g_0'),g_1)) = w_0(g_0) \Phi (s((g_0,g_0'),g_1)) w_0(g_0') ^{-1}
\]
for a.e.\ $((g_0,g_0'),g_1)\in \mc{K}$. Therefore, by Claim \ref{claim:Gwm} and Lemma \ref{lem:Gbasic}, there is a measurable map $f:\mc{K}_0^0 \rightarrow L$ such that for a.e.\ $((x_0,x_0'),x_1)\in \mc{K}^0$ we have $\Phi ((x_0,x_0'),x_1) = f(x_0,x_0')$, i.e.,
\[
F(x_0,x_1) = f(x_0,x_0')F(x_0',x_1).
\]
Since for $y\in \mc{H}^0$ we have $\mu _{\mc{K}^0_y}= \mu _{(\mc{G}^0_0)_y}\otimes \mu _{(\mc{G}^0_0)_y}\otimes \mu _{(\mc{G}^0_1)_y}$, by applying Fubini's theorem and an appropriate selection theorem (e.g., \cite[Theorem 18.1]{Ke95}), we can find a measurable section $t: \mc{H}^0\rightarrow \mc{G}_0^0$ for $q_0^0:\mc{G}_0^0 \ra\mc{H}^0$ such that for a.e.\ $y\in \mc{H}^0$, for $\mu _{(\mc{G}_0\otimes \mc{G}_1)^0_y}$-a.e.\ $(x_0,x_1)\in (\mc{G}_0\otimes _{\mc{H}}\mc{G}_1)^0_y$ we have $F(x_0,x_1)=f(x_0,t(y))F(t(y), x_1)$. Let $\varphi _0 : \mc{G}_0^0\rightarrow L$ be the map $\varphi _0(x_0) := f(x_0,t(q_0(x_0)))^{-1}$ and let $\varphi _1 :\mc{G}_1^0\rightarrow L$ be the map $\varphi _1(x_1) := F(t(q_1(x_1)),x_1 )$, so that $F(x_0,x_1)= \varphi _0(x_0)^{-1}\varphi _1(x_1)$ for a.e.\ $(x_0,x_1)\in (\mc{G}_0\otimes _{\mc{H}}\mc{G}_1)^0$. Then, by \eqref{eqn:GFw}, for a.e.\ $(g_0,g_1)\in \mc{G}_0\otimes _{\mc{H}}\mc{G}_1$ we have
\begin{equation}\label{eqn:Gfub}
\varphi _0(r(g_0))w_0(g_0) \varphi _0 (s(g_0))^{-1} = \varphi _1 (r(g_1))w_1(g_1)\varphi _1(s(g_1))^{-1} .
\end{equation}
Thus, for a.e.\ $h\in \mc{H}$, for $\mu _{(\mc{G}\otimes _{\mc{H}}\mc{G})_h} = \mu _{(\mc{G}_0)_h}\otimes \mu _{(\mc{G}_1)_h}$ a.e.\ $(g_0,g_1)\in (\mc{G}_0\otimes _{\mc{H}} \mc{G}_1)_h=(\mc{G}_0)_h\otimes (\mc{G}_1)_h$, \eqref{eqn:Gfub} holds. Fubini's Theorem then implies that there is some $w(h)\in L$, along with a $\mu _{(\mc{G}_0)_h}$-conull subset $C_{0,h}$ of $(\mc{G}_0)_h$ and a $\mu _{(\mc{G}_1)_h}$-conull subset $C_{1,h}$ of $(\mc{G}_1)_h$ such that the left and right hand sides of \eqref{eqn:Gfub} are equal to $w(h)$ for all $g_0\in C_{0,h}$ and $g_1\in C_{1,h}$. It then follows that, after discarding a $\mu _{\mc{H}}$-null set, the assignment $h\mapsto w(h)$ is a homomorphism, and satisfies the conclusion of the lemma.
\end{proof}

\subsection{Bernoulli extensions}

\begin{definition}
Let $K$ be a standard probability space. We let $\EuScript{B}_K(\mc{G})$ denote the translation groupoid associated to the Bernoulli action $\beta ^{\mc{G}}_{K}$ of $\mc{G}$ with base $K$, i.e., $\EuScript{B}_K (\mc{G}) := \mc{G}\ltimes K^{\otimes \mc{G}}$. The projection map $\EuScript{B}_K(\mc{G})\rightarrow \mc{G}$ is a groupoid extension called the {\bf Bernoulli extension of $\mc{G}$ with base $K$}. We write $\EuScript{B}(\mc{G})$ for $\EuScript{B}_{[0,1]}(\mc{G})$.
\end{definition}

\begin{remark}\label{rem:translbern}
If $\mc{G}=\mc{H}\ltimes X$ is a translation groupoid associated to a p.m.p.\ action of $\mc{H}$ then, by Proposition \ref{prop:BernIsom} below and Remark \ref{rem:translext}, the $\mc{G}$-extensions $\EuScript{B}_K(\mc{H}\ltimes X)\rightarrow \mc{H}\ltimes X$ and $\mc{H}\ltimes (X\otimes _{\mc{H}^0} K^{\otimes \mc{H}})\rightarrow \mc{H}\ltimes X$ are isomorphic. In particular, for a group $H$, we have an isomorphism of groupoids $\EuScript{B}_K(H\ltimes X) \cong H\ltimes (X\otimes K^H)$, where $H\cc X\otimes K^H$ is the product of $H\cc X$ with the Bernoulli action of $H$ over the base space $K$.
\end{remark}

If $\mc{G}\xra{p}\mc{H}$ is a groupoid extension, then we obtain an extension $\EuScript{B}_K(\mc{G})\xra{\EuScript{B}_K(p)}\EuScript{B}_K(\mc{H})$ as follows: for each $x\in \mc{G}^0$ let $p^x$ denote the restriction of $p$ to $x\mc{G}$, so that $p^x : x\mc{G}\rightarrow p(x)\mc{H}$ is a bijection for a.e.\ $x\in \mc{G}^0$. For such $x$, define $\EuScript{B}_K(p)^x : K^{x\mc{G}}\rightarrow K^{p(x)\mc{H}}$ by $\EuScript{B}_K(p)^x (f)(h) = f((p^x)^{-1}(h))$, for $f\in K^{p(x)\mc{H}}$, $h\in p(x)\mc{H}$. The map $\EuScript{B}_K(p)^x$ is measure preserving, and defining $\EuScript{B}_K(p)(g,f):= (p(g),\EuScript{B}_K(p)^{s(g)}(f))$ for $(g,f)\in \EuScript{B}_K(\mc{G})$, makes $\EuScript{B}_K(p)$ a groupoid extension. Moreover, we have the following commuting square of groupoid extensions
\begin{equation}\label{eqn:Bsquare}
\xymatrix{
\EuScript{B}_K(\mc{G})\ar[d]\ar[r]^{\EuScript{B}_K(p)}    &\EuScript{B}_K(\mc{H})\ar[d] \\
\mc{G}\ar[r]^{p}                        &\mc{H}
}
\end{equation}

\begin{proposition}\label{prop:functor}
Let $K$ be a standard probability space. The assignment which takes each discrete p.m.p.\ groupoid $\mc{G}$ to the associated Bernoulli extension $\EuScript{B}_K(\mc{G})\rightarrow \mc{G}$, and which takes each groupoid extension $\mc{G}\xra{p}\mc{H}$ to the associated commuting square \eqref{eqn:Bsquare}, is a functor from the category $\bm{\mathrm{DPG}}$ of discrete p.m.p.\ groupoids, to the category $\bm{\mathrm{DPG ^2}}$, of groupoid extensions.
\end{proposition}

The next proposition generalizes the fact that an orbit equivalence between essentially free p.m.p.\ actions $G\cc X$ and $H\cc Y$ of countable groups induces an orbit equivalence of the actions $G\cc X \otimes K^G$ and $H\cc Y\otimes K^H$.

\begin{proposition}\label{prop:BernIsom}
Let $\mc{G}\xra{p} \mc{H}$ be an extension of of discrete p.m.p.\ groupoids. Then the $\mc{G}$-extensions $\EuScript{B}_K(\mc{G})\xra{q}\mc{G}$ and $\mc{G}\otimes _{\mc{H}}\EuScript{B}_K(\mc{H})\rightarrow \mc{G}$ are isomorphic, with an isomorphism being given by the map $q\otimes \EuScript{B}_K(p): (g,f)\mapsto (q(g,f), \EuScript{B}_K(p)(g,f))=(g,\EuScript{B}_K(p)(g,f))$,
\begin{equation}\label{eqn:BernIsom}
\xymatrix{
\EuScript{B}_K(\mc{G})\ar[d]_q \ar[rr]_{\cong}^{q\otimes \EuScript{B}_K(p)} & \                                     & \mc{G}\otimes _{\mc{H}}\EuScript{B}_K(\mc{H}) \ar[d] \\
\mc{G} \ar[rr]^{\mathrm{id}}               & \                                     & \mc{G}
}
\end{equation}
\end{proposition}

\begin{proof}
This is a straightforward verification.
\end{proof}

In the case of equivalence relations, the next lemma follows from \cite[Theorem 3.3]{BHI15}.

\begin{lemma}\label{lem:Binfl}
Let $\mc{G}$ be a discrete p.m.p.\ groupoid and let $A$ be a non-null measurable subset of $\mc{G}^0$.  Then the $\mc{G}_A$-extensions $\EuScript{B}(\mc{G}_A)\rightarrow \mc{G}_A$ and $\EuScript{B}(\mc{G})_A \rightarrow \mc{G}_A$ are isomorphic.
\end{lemma}

\begin{proof}
Let $X = r(\mc{G}A)$. Then $X$ is a $\mc{G}$-invariant measurable subset of $\mc{G}^0$, so the extensions $\EuScript{B}(\mc{G})_X\rightarrow \mc{G}_X$ and $\EuScript{B}(\mc{G}_X)\rightarrow \mc{G}_X$ coincide. Therefore, after replacing $\mc{G}$ by $\mc{G}_X$ if necessary, we may assume that $A$ is a complete unit section for $\mc{G}$. By Lemma \ref{lem:Biso}, the actions $\beta ^{\mc{G}}_{[0,1]} : \mc{G} \cc [0,1]^{\otimes \mc{G}}$ and $\beta ^{\mc{G}A}_{[0,1]} : \mc{G}\cc [0,1]^{\otimes \mc{G}A}$ are isomorphic, and hence the $\mc{G}$-extensions $\EuScript{B}(\mc{G})\rightarrow \mc{G}$ and $\mc{G}\ltimes [0,1]^{\otimes \mc{G}A}\rightarrow \mc{G}$ are isomorphic. Therefore, the $\mc{G}_A$-extensions $\EuScript{B}(\mc{G})_A \rightarrow \mc{G}_A$ and $(\mc{G}\ltimes [0,1]^{\otimes \mc{G}A})_A\rightarrow \mc{G}_A$ are isomorphic. The latter $\mc{G}_A$ extension coincides precisely with $\EuScript{B}(\mc{G}_A)\rightarrow \mc{G}_A$.
\end{proof}

The relative weak mixing assumption in Lemma \ref{lem:GFurman} will manifest in \S\ref{sec:ME} through the following well-known lemma.

\begin{lemma}\label{lem:Bwm}
Let $\mc{G}$ be a discrete p.m.p.\ groupoid which is aperiodic, i.e., with $x\mc{G}$ infinite for a.e.\ $x\in \mc{G}^0$, and let $K$ be a standard probability space. Then the Bernoulli extension $\EuScript{B}_K(\mc{G})\rightarrow \mc{G}$ is relatively weakly mixing.
\end{lemma}

\begin{proof}
By Proposition \ref{prop:clear} and Remark \ref{rem:translext}, the $\mc{G}$-extensions $\EuScript{B}_K(\mc{G})\otimes _{\mc{G}} \EuScript{B}_K(\mc{G})\rightarrow \mc{G}$ and $\EuScript{B}_{K^2}(\mc{G})\rightarrow \mc{G}$ are isomorphic, hence it suffices to show that the extension $\EuScript{B}_K(\mc{G})\rightarrow \mc{G}$ is relatively ergodic. In terms of the Bernoulli action $\beta ^{\mc{G}}_K:\mc{G}\cc K^{\otimes \mc{G}}$, this means the following: given a measurable subset $A$ of $K^{\otimes \mc{G}}$ which is $\beta ^{\mc{G}}_K$-invariant, we must show that for a.e.\ $x\in \mc{G}^0$ the fiber $A_x\subseteq K^{x\mc{G}}$ is either null or conull.

For each $x\in \mc{G}^0$ let $\mu _x:= \mu _{K}^{x\mc{G}}$ denote the measure on the fiber $K^{x\mc{G}}$, and for each finite $Q\subseteq x\mc{G}$ let $\pi _Q : K^{x\mc{G}}\rightarrow K^{Q}$ denote the restriction map $\pi _Q(f):= f|Q$.

Fix $\epsilon >0$. For each $x\in \mc{G}^0$ let $n(x)$ be the least natural number $n$ such that there exists some subset $Q\subseteq x\mc{G}$ of cardinality $|Q|=n$ having the following property: $(*)$ there exists some $\pi _Q$-measurable subset $B$ of $K^{x\mc{G}}$ with $\mu _x(B\triangle A_x)\leq \epsilon \mu _x(A_x)$. We can then find a measurable selection $x\mapsto Q^x\subseteq x\mc{G}$ of such a set satisfying $(*)$, with $|Q^x| = n(x)$, so that the set $Q:=\bigcup _{x\in \mc{G}^0}Q^x \subseteq \mc{G}$ is measurable with $xQ=Q^x$ for $x\in \mc{G}^0$. Since $A$ is $\beta ^\mc{G}_K$-invariant, the function $x\mapsto n(x)$ is $\mc{G}$-invariant, hence for each $n\in \N$ the set $D_n: = \{ x\in \mc{G}^0 \, : \, n(x)=n \}$ is $\mc{G}$-invariant, which implies that
\[
\int _{D_n}|Qx|\, d\mu _{\mc{G}^0} =\int _{D_n}|xQ| \, d\mu _{\mc{G}^0} = n\mu _{\mc{G}^0}(D_n)<\infty  .
\]
In particular, the set $Qx$ is finite for a.e.\ $x\in \mc{G}^0$. It follows that $xQQ^{-1}$ is a finite subset of the infinite set $x\mc{G}$ for a.e.\ $x\in \mc{G}^0$. Fix such an $x$ and choose some $g\in x\mc{G}$ which does not belong to the finite set $xQQ^{-1}$, and hence satisfies $gQ\cap Q=\emptyset$. Let $y:=s(g)$ and find subsets $B_y\subseteq K^{y\mc{G}}$ and $B_x \subseteq K^{x\mc{G}}$ which are $\pi _{yQ}$-measurable and $\pi _{xQ}$-measurable respectively and satisfy $\mu _y(B_y\triangle A_y)\leq \epsilon \mu _y(A_y)$ and $\mu _x(B_x\triangle A_x)\leq \epsilon \mu _x (A_x)$. Since $A$ is $\beta ^{\mc{G}}_K$-invariant we have $\mu _y (A_y)=\mu _x (A_x)$ and
\begin{align}
\label{eqn:Ax} \mu _x(A_x) =\mu _{x}(\beta ^{\mc{G}}(g)A_y\cap A_x) &\leq \mu _x(\beta ^{\mc{G}}(g)B_y\cap B_x) +2\epsilon \mu _x(A_x) .
\end{align}
Since $gQ\cap Q=\emptyset$, the subsets $\beta ^{\mc{G}}(g)B_y$ and $B_x$ of $K^{x\mc{G}}$ are independent for the measure $\mu _x$, and hence
\begin{align*}
\mu _x (\beta ^{\mc{G}}(g)B_y \cap B_x) &= \mu _y(B_y)\mu _x (B_x) \leq (1+\epsilon )^2 \mu _x(A_x)^2 .
\end{align*}
Combining this with \eqref{eqn:Ax} shows that $\mu _x(A_x)\leq \frac{(1+\epsilon )^2 }{1-2\epsilon} \mu _x(A_x)^2$. This inequality holds for a.e.\ $x\in \mc{G}^0$. Since $\epsilon >0$ was arbitrary, we conclude that for a.e. $x\in \mc{G}^0$ we have $\mu _x(A_x)\leq \mu _x(A_x)^2$ and hence $A_x$ is either null or conull, as was to be shown.
\end{proof}

\begin{lemma}\label{lem:Bsubg}
Let $\mc{G}$ be a discrete p.m.p.\ groupoid, let $A$ be a non-null measurable subset of $\mc{G}^0$ and let $\mc{H}$ be a measurable subgroupoid of $\mc{G}_A$ with $\mc{H}^0=A$. Let $p:\EuScript{B}(\mc{G})\rightarrow \mc{G}$ be the Bernoulli extension of $\mc{G}$ with base space $[0,1]$, and let $\mc{H}':=p^{-1}(\mc{H})$. Then the $\mc{H}$-extension $p|_{\mc{H}'}: \mc{H}'\rightarrow \mc{H}$ is isomorphic to the Bernoulli $\mc{H}$-extension $\EuScript{B}(\mc{H})\rightarrow \mc{H}$.
\end{lemma}

\begin{proof}
By Lemma \ref{lem:Binfl}, we may assume that $A=\mc{G}^0$, and hence that $\mc{H}^0=\mc{G}^0$. In terms of p.m.p.\ actions of $\mc{H}$, we must show that the p.m.p.\ actions $\beta ^{\mc{G}}|_{\mc{H}} : \mc{H} \cc [0,1]^{\otimes \mc{G}}$ and $\beta ^{\mc{H}} : \mc{H}\cc [0,1]^{\otimes \mc{H}}$ are isomorphic, where $\beta ^{\mc{G}}|_{\mc{H}}$ denotes the restriction to $\mc{H}$ of the standard Bernoulli action of $\mc{G}$.

By the Lusin-Novikov Uniformization Theorem we may find a sequence $\tau _0, \tau _1,\dots$ of measurable bisections of $\mc{G}$ with $\mc{G}= \bigcup _{i\in \N} \tau _i$. Let $\sigma _{n} := \tau _n \setminus \bigcup _{i<n}\mc{H}\tau _i$ and put $\Sigma := \{ \sigma _n  \} _{n\in \N}$. Then $\Sigma$ is a countable collection of measurable bisections of $\mc{G}$, and the sets $\mc{H}\sigma$, $\sigma \in \Sigma$, form a partition of $\mc{G}$. For each $\sigma \in \Sigma$ the map $\rho _{\sigma} : \mc{H}\sigma \rightarrow \mc{H}r(\sigma)$, defined by $\rho _{\sigma}(g) := g\sigma ^{-1}$, gives an isomorphism of measure preserving discrete $\mc{H}$-actions, from the left translation action $\mc{H}\cc \mc{H}\sigma$, to the left translation action $\mc{H}\cc \mc{H}r(\sigma )$. Therefore, the left translation action of $\mc{H}$ on $\mc{G}$ is isomorphic to the action of $\mc{H}$ on the disjoint union $V:= \bigsqcup _{\sigma \in \Sigma}\mc{H}r(\sigma )$, and hence we have isomorphisms $\beta ^{\mc{G}}|_{\mc{H}} \cong \beta ^{V}_{[0,1]}$ of p.m.p.\ actions of $\mc{H}$. By Proposition \ref{prop:clear} and Lemma \ref{lem:Biso} we have $\beta ^V_{[0,1]} \cong \bigotimes _{\sigma \in \Sigma}\beta ^{\mc{H}r(\sigma)}_{[0,1]}\cong \bigotimes _{\sigma\in\Sigma}\beta ^{\mc{H}}_{[0,1]} \cong \beta ^{\mc{H}}$.
\end{proof}

\begin{proposition}\label{prop:Bprinc}
Let $\mc{G}$ be an aperiodic discrete p.m.p.\ groupoid. Let $K$ be a standard probability space and assume that $\mu _K$ is not a point mass. Then the discrete p.m.p.\ groupoid $\EuScript{B}_K(\mc{G})$ is principal.
\end{proposition}

\begin{proof}
This is equivalent to saying that the action $\beta ^{\mc{G}}_K$ is essentially free, i.e., for a.e. $g\in \mc{G}$ with $g\not\in \mc{G}^0$, for a.e.\ $f\in K^{s(g)\mc{G}}$ we have $\beta ^{\mc{G}}_K(g)f \neq f$. This is clear if $s(g)\neq r(g)$, so we may assume that $s(g)=r(g)=x$. Since $\mc{G}$ is aperiodic we may assume that $x\mc{G}$ is infinite. Since $g\not\in \mc{G}^0$, the permutation $x\mc{G} \rightarrow x\mc{G}$, $h\mapsto g^{-1}h$, of $x\mc{G}$ has no fixed points, so we may find an infinite subset $I$ of $x\mc{G}$ such that $g^{-1}I\cap I = \emptyset$. If $f\in K^{x\mc{G}}$ and $\beta ^{\mc{G}}_K(g)f=f$ then, in particular, we have $f(g^{-1}h)=f(h)$ for all $h\in I$. The set of $f\in K^{x\mc{G}}$ where this occurs is null for the product measure $\mu _K^{x\mc{G}}$, since $\mu _K$ is not a point mass.
\end{proof}

\subsection{Measure equivalence of discrete p.m.p.\ groupoids}

\begin{definition}\label{def:MEgroupoid}
Let $\mc{G}$ and $\mc{H}$ be discrete p.m.p.\ groupoids.
\begin{enumerate}
\item We say that $\mc{G}$ and $\mc{H}$ are {\bf extension equivalent}, denoted $\mc{G}\simeq _{\mathrm{ext}}\mc{H}$, if there exists a discrete p.m.p.\ groupoid $\mc{K}$ which is both an extension $\mc{K}\rightarrow \mc{G}$, of $\mc{G}$, and an extension $\mc{K}\rightarrow \mc{H}$, of $\mc{H}$.
\item We say that $\mc{G}$ and $\mc{H}$ are {\bf reduction equivalent}, denoted $\mc{G}\simeq _{\mathrm{red}}\mc{H}$, if there exist measurable complete unit sections $A\subseteq \mc{G}^0$ and $B\subseteq \mc{H}^0$, such that $\mc{G}_A$ and $\mc{H}_B$ are isomorphic.
\item We say that $\mc{G}$ and $\mc{H}$ are {\bf measure equivalent} (or {\bf extension-reduction equivalent}), denoted $\mc{G}\simeq _{\mathrm{ME}}\mc{H}$, if there exist extensions $\wt{\mc{G}}\rightarrow \mc{G}$ and $\wt{\mc{H}}\rightarrow \mc{H}$ such that $\wt{\mc{G}}$ and $\wt{\mc{H}}$ are reduction equivalent.
\end{enumerate}
\end{definition}

We note that, by Proposition \ref{prop:ExtActCorr}, $\mc{G}$ and $\mc{H}$ are measure equivalent if and only if they are stably orbit equivalent in the sense that there exist p.m.p.\ actions $\mc{G}\cc Y$ and $\mc{H}\cc Z$ such that the translation groupoids $\mc{G}\ltimes Y$ and $\mc{H}\ltimes Z$ are reduction equivalent. It follows that the definition of measure equivalence for groupoids given here is consistent with the definition of measure equivalence for groups given in the introduction.

\begin{proposition}
Each of the relations $\simeq _{\mathrm{ext}}$, $\simeq _{\mathrm{red}}$, and $\simeq _{\mathrm{ME}}$, is an equivalence relation on the class of discrete p.m.p.\ groupoids. In addition, $\simeq _{\mathrm{ext}}$ is the transitive closure of the relation
\[
R_{\mathrm{ext}}= \{ (\mc{G},\mc{H} ) \, : \, \text{there exists either an extension }\mc{G}\rightarrow \mc{H}\text{ or an extension } \mc{H}\rightarrow \mc{G} \} ,
\]
and $\simeq _{\mathrm{ME}}$ is the transitive closure of the union of $R_{\mathrm{ext}}$ and $\simeq _{\mathrm{red}}$. 
\end{proposition}

\begin{proof}
Suppose first that $\mc{H}_0\simeq _{\mathrm{red}}\mc{G}$ and $\mc{G} \simeq _{\mathrm{red}}\mc{H}_1$, so that for $i=0,1$ there exist complete unit sections $B_i\subseteq \mc{H}^0_i$ and $A_i\subseteq \mc{G}^0$ with $(\mc{H}_i)_{B_i}\cong \mc{G}_{A_i}$. The assumption that $A_i$ is a complete unit section of $\mc{G}$ is equivalent, modulo a $\mu _{\mc{G}}$-null set, to $A_i$ having positive measure within almost every ergodic component of $\mc{G}$. We may therefore find complete unit sections $A_0'\subseteq A_0$ and $A_1'\subseteq A_1$ such that $A_0'$ and $A_1'$ have the same positive measure within almost every ergodic component of $\mc{G}$, and hence $\mc{G}_{A_0'}\cong \mc{G}_{A_1'}$ by Proposition \ref{prop:IsoRed}. Then, for $i=0,1$, the isomorphisms $(\mc{H}_i)_{B_i}\cong \mc{G}_{A_i}$ yield complete unit sections $B_i'\subseteq B_i$ such that $(\mc{H}_i)_{B_i'}\cong \mc{G}_{A_i'}$, and therefore $(\mc{H}_0)_{B_0'}\cong (\mc{H}_1)_{B_1'}$, i.e., $\mc{H}_0\simeq _{\mathrm{red}}\mc{H}_1$.

Since $\simeq _{\mathrm{ext}}$ clearly lies between $R_{\mathrm{ext}}$ and its transitive closure, to see that $\simeq _{\mathrm{ext}}$ is the transitive closure of $R_{\mathrm{ext}}$ it is enough to show that $\simeq _{\mathrm{ext}}$ is an equivalence relation. Similarly, it is enough to show that $\simeq _{\mathrm{ME}}$ is an equivalence relation. For $\simeq _{\mathrm{ext}}$, if $\mc{K}_0$ is an extension of both $\mc{G}_0$ and $\mc{H}$, and $\mc{K}_1$ is an extension of both $\mc{H}$ and $\mc{G}_1$, then $\mc{K}_0\otimes _{\mc{H}}\mc{K}_1$ is an extension of both $\mc{G}_0$ and $\mc{G}_1$; it follows that $\simeq _{\mathrm{ext}}$ is an equivalence relation.

Assume now that $\mc{G}\simeq _{\mathrm{ME}}\mc{H}$ and $\mc{H}\simeq _{\mathrm{ME}} \mc{K}$, so that there are extensions $\mc{G}_0\rightarrow \mc{G}$ and $\mc{H}_0\rightarrow \mc{H}$ along with complete unit sections $A_0$ of $\mc{G}_0$ and $B_0$ of $\mc{H}_0$ such that $(\mc{G}_0)_{A_0}\cong (\mc{H}_0)_{B_0}$, and also there are extensions $\mc{H}_1\rightarrow \mc{H}$ and $\mc{K}_1\rightarrow \mc{K}$ along with complete unit sections $B_1$ of $\mc{H}_1$ and $C_1$ of $\mc{K}_1$ such that $(\mc{H}_1)_{B_1}\cong (\mc{K}_1)_{C_1}$. Let $\wt{\mc{H}}$ be the relatively independent product, $\wt{\mc{H}}:= \mc{H}_0\otimes _{\mc{H}}\mc{H}_1$, of $\mc{H}_0$ and $\mc{H}_1$ over $\mc{H}$. Then the composition $\wt{\mc{H}}_{B_0}\rightarrow (\mc{H}_0)_{B_0}\xra{\cong} (\mc{G}_0)_{A_0}$ realizes $\wt{\mc{H}}_{B_0}$ as an extension of $(\mc{G}_0)_{A_0}$, so by Proposition \ref{prop:extact} there exists an extension $\wt{\mc{G}}\rightarrow \mc{G}_0$ such that $\wt{\mc{G}}_{A_0}\cong \wt{\mc{H}}_{B_0}$. Likewise, there exists an extension $\wt{\mc{K}}\rightarrow \mc{K}_1$ such that $\wt{\mc{H}}_{B_1}\cong \wt{\mc{K}}_{C_1}$. Therefore $\wt{\mc{G}}\simeq _{\mathrm{red}}\wt{\mc{H}}\simeq _{\mathrm{red}}\wt{\mc{K}}$, and hence $\mc{G}\simeq _{\mathrm{ME}}\mc{K}$.
\end{proof}

\section{ME-invariance}\label{sec:ME}

\begin{definition}\label{def:GBCS}
Let $\mathscr{C}$ be a class of Polish groups. We say that an extension $\mc{G}\rightarrow \mc{H}$ of discrete p.m.p.\ groupoids is {\bf relatively $\mathscr{C}$-superrigid} if every measurable homomorphism $w: \mc{G} \rightarrow L$ taking values in a group $L\in \mathscr{C}$ is equivalent to a homomorphism which descends to $\mc{H}$. We say that a discrete p.m.p.\ groupoid $\mc{H}$ is {\bf Bernoulli $\mathscr{C}$-superrigid} if for every extension $\mc{G}\rightarrow \mc{H}$ of $\mc{H}$, the Bernoulli extension $\EuScript{B}(\mc{G})\rightarrow \mc{G}$ is relatively $\mathscr{C}$-superrigid.
\end{definition}

\begin{proposition}\label{prop:groupBCS} Let $H$ be a countable group and let $\mathscr{C}$ be a class of Polish groups. Then $H$ is Bernoulli $\mathscr{C}$-superrigid in the sense of Definition \ref{def:GBCS} if and only if $H$ is Bernoulli $\mathscr{C}$-superrigid in the sense of Definition \ref{def:BCS}.
\end{proposition}

\begin{proof} This is primarily an exercise in lexicography. Explicitly:
\begin{itemize}
\item[(i)] Measurable $L$-valued cocyles of a p.m.p.\ action $H\cc X$ of $H$ are the same as measurable homomorphisms from the translation groupoid $H\ltimes X$ into $L$, with two cocycles being cohomologous if and only if they are equivalent as homomorphisms.
\item[(ii)] By (i), an extension $X\rightarrow Y$ of p.m.p.\ actions of $H$ is relatively $\mathscr{C}$-superrigid in the sense of Definition \ref{def:BCS} if and only if the associated extension $H\ltimes X \rightarrow H\ltimes Y$ of translation groupoids is relatively $\mathscr{C}$-superrigid in the sense of Definition \ref{def:GBCS}.
\item[(iii)] If $H\cc Y$ is a p.m.p.\ action of $H$ then by Proposition \ref{prop:BernIsom} and Remark \ref{rem:translext} we have isomorphisms of $(H\ltimes Y)$-extensions $\EuScript{B}(H\ltimes Y)\cong (H\ltimes Y )\otimes _H \EuScript{B}(H)\cong H\ltimes (Y\otimes [0,1]^H )$. Thus, by (ii), the groupoid extension $\EuScript{B}(H\ltimes Y) \rightarrow H\ltimes Y$ is relatively $\mathscr{C}$-superrigid if and only if the extension $Y\otimes [0,1]^H \rightarrow Y$ of p.m.p.\ actions of $H$ is relatively $\mathscr{C}$-superrigid.
\item[(iv)] By Proposition \ref{prop:ExtActCorr}, every $H$-extension $\mc{G}\rightarrow H$ is isomorphic to an $H$-extension $H\ltimes Y\rightarrow H$ associated to a translation groupoid of some p.m.p.\ action $H\cc Y$ of $H$, and hence the associated Bernoulli extensions $\EuScript{B}(\mc{G})\rightarrow \mc{G}$ and $\EuScript{B}(H\ltimes Y )\rightarrow H\ltimes Y$ are isomorphic by Proposition \ref{prop:functor}.
\end{itemize}
The proposition follows from the compilation of (i)-(iv).
\end{proof}

\begin{lemma}\label{lem:srred}
Let $\mathscr{C}$ be a class of Polish groups. Let $\mc{H}$ be a discrete p.m.p.\ groupoid and let $A$ be a measurable complete unit section of $\mc{H}$. Then $\mc{H}$ is Bernoulli $\mathscr{C}$-superrigid if and only if $\mc{H}_{A}$ is Bernoulli $\mathscr{C}$-superrigid.
\end{lemma}

\begin{proof}
Let $\mc{G} \xra{p}\mc{H}$ be an $\mc{H}$-extension, so that $\mc{G}_A\xra{p_A}\mc{H}_A$ is an $\mc{H}_A$-extension. We abuse notation and identify $A$ with its preimage $p^{-1}(A)\subseteq \mc{G}^0$ under $p$. Then $A$ is a measurable complete unit section for $\mc{G}$, so by Lemma \ref{lem:Binfl}, the $\mc{G}_A$-extensions $\EuScript{B}(\mc{G}_A)\rightarrow \mc{G}_A$ and $\EuScript{B}(\mc{G})_A \rightarrow \mc{G}_A$ are isomorphic.
Thus, the extension $\EuScript{B}(\mc{G}_A)\rightarrow \mc{G}_A$ being relatively $\mathscr{C}$-superrigid is equivalent to the extension $\EuScript{B}(\mc{G})_A \rightarrow \mc{G}_A$ being relatively $\mathscr{C}$-superrigid, and by (2) of Proposition \ref{prop:restrict}, this is equivalent to the extension $\EuScript{B}(\mc{G})\rightarrow \mc{G}$ being relatively $\mathscr{C}$-superrigid. Since, by Proposition \ref{prop:extact}, every $\mc{H}_{A}$-extension is isomorphic to the reduction of an $\mc{H}$-extension, it follows that $\mc{H}$ is Bernoulli $\mathscr{C}$-superrigid if and only if $\mc{H}_{A}$ is Bernoulli $\mathscr{C}$-superrigid.
\end{proof}

\begin{lemma}\label{lem:srext}
Let $\mathscr{C}$ be a class of Polish groups contained in $\mathscr{G}_{\mathrm{inv}}$, and let $\mc{G}\rightarrow \mc{H}$ be an extension of discrete p.m.p.\ groupoids. Then $\mc{G}$ is Bernoulli $\mathscr{C}$-superrigid if and only if $\mc{H}$ is Bernoulli $\mathscr{C}$-superrigid.
\end{lemma}

\begin{proof}
It is clear that periodic groupoids (i.e., groupoids whose source and range maps are finite-to-one) are Bernoulli $\mathscr{C}$-superrigid, hence by breaking $\mc{H}$ into its periodic and aperiodic parts, it is enough to consider the case where $\mc{H}$ is aperiodic.

If $\mc{H}$ is Bernoulli $\mathscr{C}$-superrigid, then it is clear that $\mc{G}$ is Bernoulli $\mathscr{C}$-superrigid, since every extension of $\mc{G}$ is also an extension of $\mc{H}$.

Assume now that $\mc{G}$ is Bernoulli $\mathscr{C}$-superrigid. Let $\mc{K}\rightarrow \mc{H}$ be another extension of $\mc{H}$. We must show that the Bernoulli extension $\EuScript{B}(\mc{K})\xra{q_1}\mc{K}$ is relatively $\mathscr{C}$-superrigid. Let $\mc{G}\otimes _{\mc{H}}\mc{K}\xra{q_0}\mc{K}$ be the right projection. We then have the following commutative diagram of groupoid extensions
\[
\xymatrix{
\EuScript{B}(\mc{G}\otimes _{\mc{H}}\mc{K})\ar[d]^p\ar[r]^{\cong \ \ \ \ \ }   &(\mc{G}\otimes  _{\mc{H}}\mc{K})\otimes _{\mc{K}}\EuScript{B}(\mc{K}) \ar[d]^{p_0}\ar[r]^{\ \ \ \ \ \ \ \ \ p_1} &\EuScript{B}(\mc{K})\ar[d]^{q_1} \\
\mc{G}\otimes _{\mc{H}}\mc{K} \ar[r]^{\mathrm{id}} &\mc{G}\otimes _{\mc{H}}\mc{K} \ar[r]^{q_0}    &\mc{K}
}
\]
where the left square is given by Proposition \ref{prop:BernIsom} applied to the extension $q_0$, and $p_0$ and $p_1$ are the left and right projections respectively. Since $\mc{G}$ is Bernoulli $\mathscr{C}$-superrigid, and $\mc{G}\otimes _{\mc{H}}\mc{K}$ is an extension of $\mc{G}$, the extension $p$ is relatively $\mathscr{C}$-superrigid. Therefore, the extension $p_0$ is relatively $\mathscr{C}$-superrigid as well. In addition, by Lemma \ref{lem:Bwm}, the Bernoulli extension $q_1$ is relatively weakly mixing.

Let $L\in \mathscr{C}$ and let $w_1 : \EuScript{B}(\mc{K})\rightarrow L$ be a measurable homomorphism. Since $p_0$ is relatively $\mathscr{C}$-superrigid, the homomorphism $w_1\circ p_1$ is equivalent to a homomorphism of the form $w_0\circ p_0$, where $w_0$ is some measurable homomorphism from $\mc{G}\otimes _{\mc{H}}\mc{K}$ into $L$. Therefore, by Lemma \ref{lem:GFurman}, $w_1$ is equivalent to a homomorphism which descends through $q_1$ to $\mc{K}$, as was to be shown.
\end{proof}

Theorem \ref{thm:main} is now a consequence of the following generalization to discrete p.m.p.\ groupoids.

\begin{theorem}\label{thm:Gmain}
Let $\mathscr{C}$ be a class of Polish groups contained in the class $\mathscr{G}_{\mathrm{inv}}$, and let $\mc{G}$ and $\mc{H}$ be discrete p.m.p.\ groupoids which are measure equivalent. Then $\mc{G}$ is Bernoulli $\mathscr{C}$-superrigid if and only if $\mc{H}$ is Bernoulli $\mathscr{C}$-superrigid.
\end{theorem}

\begin{proof}
This is immediate from Lemma \ref{lem:srred} and Lemma \ref{lem:srext}.
\end{proof}

\section{Proof of Corollary \ref{cor:lattice}} We record the following extension of Popa's cocycle superrigidity theorem for product groups \cite{Pop08} to the setting of ergodic discrete p.m.p.\ groupoids.

\begin{theorem}\label{thm:prodGroupoid}
Let $\mc{G}$ and $\mc{H}$ be ergodic discrete p.m.p.\ groupoids. Assume that $\mc{G}$ is nonamenable and $\mc{H}$ is aperiodic. Then the independent product $\mc{G}\otimes \mc{H}$ is Bernoulli $\mathscr{U}_{\mathrm{fin}}$-superrigid.
\end{theorem}

The case where $\mc{G}=G\ltimes X$ and $\mc{H}=H\ltimes Y$ are translation groupoids associated to ergodic p.m.p.\ actions of countable groups $G$ and $H$ follows immediately from \cite{Pop08}. One may give a proof of Theorem \ref{thm:prodGroupoid} in general which is essentially the same as the proof from \cite{Pop08}. For convenience, we will instead give a different proof of how to deduce Theorem \ref{thm:prodGroupoid} directly from the case handled in \cite{Pop08}, using Theorem \ref{thm:Gmain} and \cite[Theorem A]{BHI15}, along with the next lemma, whose proof is just a groupoid version of the proof of \cite[Lemma 3.5]{Fu07} (see also \cite[Lemma 3.6]{Pop07}).

In what follows, a countable collection $\Phi$ of measurable bisections of $\mc{G}$ is said to {\bf generate} $\mc{G}$ if for a.e.\ $g\in \mc{G}$ there exists a finite sequence $\phi _n,\dots , \phi _1\in \Phi$ such that $g\in \phi _n\cdots \phi _1$. We say that an aperiodic measurable subgroupoid $\mc{H}$ of a discrete p.m.p.\ groupoid $\mc{G}$ is {\bf q-normal} in $\mc{G}$ if there exists a countable collection $\Phi$ of measurable bisections of $\mc{G}$, which generates $\mc{G}$, such that for each $\phi \in \Phi$ the groupoid $\phi ^{-1}\mc{H}\phi \cap \mc{H}$ is an aperiodic subgroupoid of $\mc{G}_{s(\phi )}$ (i.e., $\phi ^{-1}\mc{H}\phi \cap \mc{H}\cap s^{-1}(x)$ is infinite for a.e.\ $x\in s(\phi )$).

\begin{lemma}\label{lem:qnorm}
Let $\mathscr{C}$ be a class of Polish groups contained in $\mathscr{G}_{\mathrm{inv}}$. Let $\mc{G}$ be a discrete p.m.p.\ groupoid and let $\mc{H}$ be an aperiodic measurable subgroupoid of $\mc{G}$ which is q-normal in $\mc{G}$. Suppose that $\mc{H}$ is Bernoulli $\mathscr{C}$-superrigid. Then $\mc{G}$ is also Bernoulli $\mathscr{C}$-superrigid.
\end{lemma}

\begin{proof}[Proof of Lemma \ref{lem:qnorm}]
Since q-normality and (by Theorem \ref{thm:Gmain}) Bernoulli $\mathscr{C}$-superrigidity both pass to extensions, it suffices to prove that, under the hypotheses of the lemma, the Bernoulli extension $p:\EuScript{B}(\mc{G})\rightarrow \mc{G}$ is relatively $\mathscr{C}$-superrigid. Let $w:\EuScript{B}(\mc{G})\rightarrow L$ be a measurable homomorphism to some $L\in \mathscr{C}$ and let $\mc{H}':= p^{-1}(\mc{H})$. Then, by Lemma \ref{lem:Bsubg}, the extension $p: \mc{H}'\rightarrow \mc{H}$ is isomorphic as an $\mc{H}$-extension to the Bernoulli extension $\EuScript{B}(\mc{H})\rightarrow \mc{H}$, so since $\mc{H}$ is Bernoulli $\mathscr{C}$-superrigid, after replacing $w$ by an equivalent homomorphism we may assume that there is a measurable homomorphism $w_0:\mc{H}\rightarrow L$ such that $w(h)=w_0(p(h))$ for all $h\in \mc{H}'$.

Let $\Phi$ be a countable collection of measurable bisections witnessing that $\mc{H}$ is q-normal in $\mc{G}$. For $\phi \in \Phi$ let $\mc{H}(\phi ) := \phi ^{-1}\mc{H}\phi \cap \mc{H}$, and let $\phi ' := p^{-1}(\phi )$, so that $\phi '$ is a bisection for $\EuScript{B}(\mc{G})$. By Lemma \ref{lem:Bsubg}, the extension $p: p^{-1}(\mc{H}(\phi ))\rightarrow \mc{H}(\phi )$ is a Bernoulli extension so by Lemma \ref{lem:Bwm}, since $\mc{H}(\phi )$ is aperiodic, this extension is relatively weakly mixing. For each $h\in p^{-1}(\mc{H}(\phi ))$ we have that $p(h), p(\phi ' h (\phi ')^{-1}) \in \mc{H}$ and hence
\[
w_0(\phi p(h) \phi ^{-1} ) = w(\phi ' h(\phi ') ^{-1}) = w(\phi ' r(h))w_0(p(h))w(\phi ' s(h))^{-1}.
\]
It now follows from Lemma \ref{lem:Gbasic} that for each $\phi \in \Phi$ there is a measurable map $f_{\phi} : s(\phi ) \rightarrow L$ such that $w(\phi 'x) = f_{\phi}(p(x))$ for a.e.\ $x\in p^{-1}(s(\phi ))=s(\phi ')$. Thus, for each $\phi \in \Phi$, almost every $g\in \phi '$ satisfies
\[
w(g) = w(\phi ' s(g))= f_{\phi}(p(s(g))) = f_{\phi } (s(p(g))).
\]
After discarding an invariant $\mu _{\Es{B}(\mc{G})^0}$-null set we may therefore assume that for all $\phi \in \Phi$ and $g\in \phi '$ we have $w(g)= f_{\phi}(s(p(g)))$. Given now $g\in \EuScript{B}(\mc{G})$, since $\Phi$ generates $\mc{G}$ we may find a finite sequence $(\phi _n, \phi _{n-1}\dots , \phi _1)$ of elements of $\Phi$ with $p(g)\in \phi _n\phi _{n-1} \cdots \phi _1$. We can moreover choose such a finite sequence so that it depends measurably on $p(g)$. We then have $g\in \phi _n'\phi _{n-1}' \cdots \phi _1'$. For $1\leq i\leq n$ let $g_i := \phi _i '\cdots \phi _1 ' s(g)$, so that $p(g_i)= \phi _i \cdots \phi _1 s(p(g))$ depends measurably on $p(g)$. Then
\begin{align*}
w(g)= w( g_n ) &= w(\phi _n ' r(g_{n-1}))w(g_{n-1}) \\
&= f_{\phi _n}(r(p(g_{n-1}))) w(g_{n-1}) =\cdots \\
\cdots &= f_{\phi _n } (r(p(g_{n-1})))f_{\phi _{n-1}}(r(p(g_{n-2})))\cdots f_{\phi _2}(r(p(g_1)))f_{\phi _1}(s(p(g))),
\end{align*}
which is a measurable function of $p(g)$. This shows that $w$ descends through $p$ to $\mc{G}$.
\end{proof}

\begin{proof}[Proof of Theorem \ref{thm:prodGroupoid}]
By Theorem \ref{thm:Gmain} it suffices to prove that some extension of $\mc{G}\otimes \mc{H}$ is Bernoulli $\mathscr{U}_{\mathrm{fin}}$-superrigid, so (after replacing both $\mc{G}$ and $\mc{H}$ by their respective Bernoulli extensions if necessary) we may assume without loss of generality that $\mc{G}$ and $\mc{H}$ are both principal groupoids, i.e., we may assume that $\mc{G}$ and $\mc{H}$ are ergodic discrete p.m.p.\ equivalence relations. Consider now the Bernoulli extension $\EuScript{B}(\mc{G})$ of $\mc{G}$, also viewed as an equivalence relation, and let $X$ denote the unit space of $\EuScript{B}(\mc{G})$.  Since $\mc{G}$ is ergodic and nonamenable, by \cite[Theorem A]{BHI15}, we may find a free ergodic p.m.p.\ action $\F _2 \cc X$, of the free group $\F _2$ on two generators, whose orbit equivalence relation $\mc{R}_{\F _2}$ is contained in $\EuScript{B}(\mc{G})$. Since $\mc{H}$ is ergodic and aperiodic we may find a free ergodic action $\Z \cc \mc{H}^0$, of $\Z$, whose orbit equivalence relation $\mc{R}_{\Z}$ is contained in $\mc{H}$. Then the independent product $\mc{R}_{\F _2} \otimes \mc{R}_{\Z }$ is Bernoulli $\mathscr{U}_{\mathrm{fin}}$-superrigid by \cite{Pop08}. Since the groupoid $\mc{R}_{\F _2}\otimes \mc{R}_{\Z }$ is q-normal in $\mc{R}_{\F _2}\otimes \mc{H}$, and $\mc{R}_{\F _2}\otimes \mc{H}$ is q-normal in $\EuScript{B}(\mc{G})\otimes \mc{H}$, it follows from Lemma \ref{lem:qnorm} that $\EuScript{B}(\mc{G})\otimes \mc{H}$ is Bernoulli $\mathscr{U}_{\mathrm{fin}}$-superrigid. Applying Theorem \ref{thm:Gmain} again we conclude that $\mc{G}\otimes \mc{H}$ is Bernoulli $\mathscr{U}_{\mathrm{fin}}$-superrigid.
\end{proof}

Let $G$ be a locally compact second countable group and let $G\cc X$ be a free p.m.p.\ action of $G$ on a standard probability space $X$. A Borel subset $Y$ of $X$ is said to be a {\bf cross section} for the action of $G$ if $G\cdot Y$ is conull and there exists a neighborhood $U$ of the identity in $G$ such that the sets $U\cdot y$, $y\in Y$, are pairwise disjoint. By \cite{Fo74}, a cross section always exists; we refer to \cite[Section 4]{KPV15} for a detailed discussion covering all of the following facts on cross sections which we will use.

Let $Y$ be a cross section for the action $G\cc X$, and denote by $\mc{R}_Y$ the restriction to $Y$ of the orbit equivalence relation associated to $G\cc X$, so that $\mc{R}_Y$ is a countable Borel equivalence relation on $Y$. If $G$ is unimodular, with Haar measure $\mu _G$, then there exists a unique $\mc{R}_Y$-invariant Borel probability measure $\mu _Y$ on $Y$ and constant $c=c_Y>0$ such that if $U$ is any neighborhood of the identity of $G$ as above, then the map $U\times Y\rightarrow X$, $(g,y)\mapsto g\cdot y$, takes the measure $(\mu _G|U) \otimes \mu _Y$ to $c \mu _X | (U\cdot Y)$. We will therefore view $\mc{R}_Y$ as a discrete p.m.p.\ groupoid whose unit space, $\mc{R}_Y^0$, we may naturally identify with the probability space $Y$. If $Z$ is any other cross section for $G\cc X$, then the discrete p.m.p.\ groupoids $\mc{R}_Y$ and $\mc{R}_Z$ are reduction equivalent in the sense of Definition \ref{def:MEgroupoid}.

\begin{proof}[Proof of Corollary \ref{cor:lattice}] Let $\Gamma$ be a lattice in $G$. Since $G=G_0\times G_1$ contains a lattice, it is unimodular, and therefore both $G_0$ and $G_1$ are unimodular as well. For $i=0,1$ let $G_i\cc X_i$ be a free ergodic p.m.p.\ action of $G_i$, and let $Y_i\subseteq X_i$ be a cross section for the action, with associated ergodic equivalence relation $\mc{R}_{Y_i}$. Then $Y_0\otimes Y_1$ is a cross section for the natural action $G\cc X_0\otimes X_1$ of $G=G_0\times G_1$ on $X_0\otimes X_1$, and the associated equivalence relation $\mc{R}_{Y_0\otimes Y_1}$ is naturally isomorphic to the independent product groupoid $\mc{R}_{Y_0\otimes Y_1} \cong \mc{R}_{Y_0}\otimes \mc{R}_{Y_1}$. The groupoid $\mc{R}_{Y_0}$ is nonamenable since the group $G_0$ is nonamenable, and likewise $\mc{R}_{Y_1}$ is aperiodic since $G_1$ is noncompact. Therefore, $\mc{R}_{Y_0\otimes Y_1}$ is Bernoulli $\mathscr{U}_{\mathrm{fin}}$-superrigid by Theorem \ref{thm:prodGroupoid}.

Consider now the free p.m.p.\ action of $G$ on $(X_0\otimes X_1 ) \otimes G/\Gamma$ given by $g\cdot ((x_0,x_1), h\Gamma ):= (g\cdot (x_0,x_1), gh\Gamma )$. The sets $Y:=(Y_0\otimes Y_1) \otimes G/\Gamma$ and $Z:= (X_0\otimes X_1 ) \otimes \{ 1_G\Gamma \}$ are both cross sections for this action, hence the groupoids $\mc{R}_Y$ and $\mc{R}_Z$ are reduction equivalent. The set $Z$ is $\Gamma$-invariant and $\mc{R}_Z$ is precisely the orbit equivalence relation generated by the action of $\Gamma$ on $Z$, hence $\mc{R}_Z$ is a groupoid extension of $\Gamma$. In addition, $\mc{R}_Y$ is a groupoid extension of $\mc{R}_{Y_0\otimes Y_1}$ via the projection map $Y\rightarrow Y_0\otimes Y_1$. The discrete p.m.p.\ groupoids $\Gamma$ and $\mc{R}_{Y_0\otimes Y_1}$ are therefore measure equivalent, and hence $\Gamma$ is Bernoulli $\mathscr{U}_{\mathrm{fin}}$-superrigid by Theorem \ref{thm:Gmain}.
\end{proof}

%

\section{Weak Pinsker entropy and orbit equivalence}

\subsection{Background and preparation}

\subsubsection{Ornstein's isomorphism theorem: from groups to groupoids} We will need the following theorem, which generalizes Ornstein's Isomorphism Theorem from $\Z$ to arbitrary infinite groups. The proof has recently been completed by Seward, building on Ornstein's original work, as well as the work of Stepin and the first author.

\begin{theorem}[Ornstein \cite{Orn70}, Stepin \cite{Step75}, Bowen \cite{Bo11b}, Seward \cite{Sew18}]\label{thm:OrnIso}
Let $G$ be a countably infinite group and let $K_0$ and $K_1$ be probability spaces with the same Shannon entropy, $H(K_0)= H(K_1)$. Then the Bernoulli shift actions $G\cc K_0^G$ and $G\cc K_1^G$ are isomorphic.
\end{theorem}

Theorem \ref{thm:OrnIso} can in fact be generalized from infinite groups to all aperiodic ergodic discrete p.m.p.\ groupoids. The case of principal groupoids is handled in \cite[Theorem 3.1]{BHI15}. The proof in the general case follows from Theorem \ref{thm:OrnIso} itself, along with the methods from \cite{Bo11b} and \cite{BHI15}, as we now show. Generalizing Stepin's definition \cite{Step76} in the case of groups, let us say that an ergodic discrete p.m.p.\ groupoid $\mc{G}$ is {\bf Ornstein} if the Bernoulli shift actions $\mc{G}\cc K_0^{\otimes \mc{G}}$ and $\mc{G}\cc K_1^{\otimes \mc{G}}$ are isomorphic whenever $K_0$ and $K_1$ are probability spaces with the same Shannon entropy. Thus, Theorem \ref{thm:OrnIso} says that every countably infinite group is Ornstein.

\begin{corollary}\label{cor:GOrnIso}
Every aperiodic ergodic discrete p.m.p.\ groupoid is Ornstein.
\end{corollary}

\begin{proof}
Let $\mc{G}$ be an aperiodic ergodic discrete p.m.p.\ groupoid. The proof breaks into two cases, according to whether or not the measure $\mu _{\mc{G}^0}$ has atoms. The case where $\mu _{\mc{G}^0}$ has atoms (Case 2 below) will be easily handled by Theorem \ref{thm:OrnIso}. When $\mu _{\mc{G}^0}$ is atomless (Case 1 below), the main fact which we use is: $(*)$ if $\mc{G}$ contains an ergodic subgroupoid $\mc{H}$ which is Ornstein, and with $\mc{H}^0=\mc{G}^0$, then $\mc{G}$ itself is Ornstein. This goes back to Stepin \cite{Step75} in the case of groups, and it was used in \cite{Bo11b} in the case of discrete p.m.p.\ equivalence relations. The proof of $(*)$ in general uses groupoid coinduction and it is a routine extension of the proof in the case of groups and equivalence relations; we refer the reader to \cite{Bo11b} and \cite[Theorem 3.1]{BHI15} for details.

Case 1: $\mu _{\mc{G}^0}$ is atomless. In this case, since $\mc{G}$ is aperiodic we may find an ergodic aperiodic principal amenable subgroupoid $\mc{H}$ of $\mc{G}$ as follows: by \cite[Proposition 9.3.2]{Zi84} we may find an ergodic aperiodic transformation $S:\mc{G}^0\rightarrow \mc{G}^0$ belonging to the full group of the equivalence relation $\mc{R}_{\mc{G}}$. We may then find a measurable bisection $\phi$ of $\mc{G}$ such that $r(\phi x)=S(x)$ for all $x\in \mc{G}^0$. Let $\mc{H}$ denote the subgroupoid of $\mc{G}$ generated by $\phi$. Then $\mc{R}_{\mc{H}}$ is the equivalence relation on $\mc{G}^0$ generated by the transformation $S$, and the map $\mc{H}\rightarrow \mc{R}_{\mc{H}}$, $h\mapsto (r(h),s(h))$, is an isomorphism of discrete p.m.p.\ groupoids, so $\mc{H}$ has the desired properties. Since $\mc{H}$ is isomorphic to a translation groupoid of a free ergodic action of $\Z$, it follows from Ornstein's Isomorphism Theorem that the groupoid $\mc{H}$ itself is Ornstein, and hence $\mc{G}$ is Ornstein by $(*)$.

Case 2: $\mu _{\mc{G}^0}$ has atoms. In this case, by ergodicity of $\mc{G}$ we may assume that $\mc{G}^0$ is finite and that $\mu _{\mc{G}^0}$ is normalized counting measure on $\mc{G}^0$. Say $\mc{G}^0$ contains exactly $n$ elements $x_0,\dots , x_{n-1}$. For each $i<n$ let $A_i := \{ x_i \}$, so that each of the sets $A_i$ is a complete unit section for $\mc{G}$, and the reduction groupoid $\mc{G}_{A_i} = x_i\mc{G}x_i$ is a group. Let $A:=A_0$. The group $\mc{G}_{A}$ is infinite since $\mc{G}$ is aperiodic, so by Theorem \ref{thm:OrnIso}, $\mc{G}_A$ is Ornstein. Let $K$ and $L$ be two probability spaces with $H(K)=H(L)$. We have the isomorphisms
\begin{equation}\label{eqn:betaKn}
\textstyle{\beta ^{\mc{G}}_K \cong \bigotimes _{i<n} \beta ^{\mc{G}A_i}_K \cong \bigotimes _{i<n}\beta ^{\mc{G}A}_K \cong \beta ^{\mc{G}A}_{K^n}} ,
\end{equation}
where the first and third isomorphisms come from Proposition \ref{prop:clear}, and the middle isomorphism comes from the isomorphisms of each of the discrete actions $\mc{G}\cc \mc{G}A_i$ with the action $\mc{G}\cc \mc{G}A$. Reducing to $\mc{G}_A$, we obtain the isomorphism $(\beta ^{\mc{G}}_K )|_{\mc{G}_A} \cong (\beta ^{\mc{G}A}_{K^n} )|_{\mc{G}_A} =\beta ^{\mc{G}_A}_{K^n}$. Likewise, we have the isomorphism $(\beta ^{\mc{G}}_L)|_{\mc{G}_A}  \cong  \beta ^{\mc{G}_A}_{L^n}$. Since $\mc{G}_A$ is Ornstein and $H(K^n)=H(L^n)$, the $\mc{G}_A$-actions $\beta ^{\mc{G}_A}_{K^n}$ and $\beta ^{\mc{G}_A}_{L^n}$ are isomorphic, and hence the reductions $(\beta ^{\mc{G}}_K)|_{\mc{G}_A}$ and $(\beta ^{\mc{G}}_L )|_{\mc{G}_A}$ are isomorphic. Proposition \ref{prop:extact} therefore implies that $\beta ^{\mc{G}}_K$ and $\beta ^{\mc{G}}_L$ are isomorphic. This shows that $\mc{G}$ is Ornstein.
\end{proof}

A corollary to this, which we will use often, is that if the groupoid Weak Pinsker entropy of $\mc{G}$ is positive, then the supremum in the definition of $h^{\mathrm{gWP}}(\mc{G})$ (given in \S\ref{subsec:gwpe}) can be taken over Bernoulli extensions of \emph{principal} groupoids.

\begin{corollary}\label{cor:principal}
Let $\mc{G}$ be an ergodic discrete p.m.p.\ groupoid and let $r\geq 0$. If $h^{\mathrm{gWP}}(\mc{G}) >r$ then $\mc{G}$ is isomorphic to $\mc{H}\ltimes K^{\otimes \mc{H}}$ for some principal ergodic discrete p.m.p.\ groupoid $\mc{H}$ and some probability space $K$ with $H(K)>r$.
\end{corollary}

\begin{proof}
Since $h^{\mathrm{gWP}}(\mc{G})>r\geq 0$, the groupoid $\mc{G}$ is aperiodic, and we may find an ergodic discrete p.m.p.\ groupoid $\mc{H}_1$ and a probability space $K_1$ with $H(K_1)>r$, such that $\mc{G}$ is isomorphic to the translation groupoid $\mc{H}_1\ltimes K_1^{\otimes \mc{H}_1}$, associated to the Bernoulli action of $\mc{H}_1$ with base $K_1$. Let $K_0$ and $K$ be probability spaces with $H(K_1) > H(K) > r$ and $H(K_0) = H(K_1)- H(K) >0$, so that $H(K_0\otimes K )= H(K_0)+H(K)=H(K_1)$. By Corollary \ref{cor:GOrnIso}, there exists a groupoid isomorphism $\mc{H}_1\ltimes K_1^{\otimes \mc{H}_1} \cong \mc{H}_1\ltimes (K_0\otimes K )^{\otimes \mc{H}_1}$. Let $\mc{H} = \mc{H}_1\ltimes K_0^{\otimes \mc{H}_1}$. Then we have isomorphisms
\[
\mc{H}\ltimes K^{\otimes \mc{H}}\cong \mc{H}\otimes _{\mc{H}_1} (\mc{H}_1 \ltimes K^{\otimes \mc{H}_1}) \cong \mc{H}_1\ltimes (K_0\otimes K )^{\otimes \mc{H}_1} \cong \mc{G}
\]
where the first isomorphism comes from Proposition \ref{prop:BernIsom} applied to the extension $\mc{H}\rightarrow \mc{H}_1$, the second isomorphism is by Remark \ref{rem:translext}, and the third was already established. The groupoid $\mc{H}$ is principal, since it is a nontrivial Bernoulli extension of $\mc{H}_1$.
\end{proof}

\subsubsection{Bernoulli superrigidity for atomic base spaces}

The next two propositions are groupoid versions of \cite[Proposition 1.2]{Fu07}.

\begin{proposition}\label{prop:furman0relWM}
Let $\mc{K}_1$, $\mc{K}_0$, and $\mc{G}$ be discrete p.m.p.\ groupoids and let $\mc{K}_1 \xra{p_1}\mc{G}$ and $\mc{K}_0\xra{p_0}\mc{G}$ be groupoid extensions. Suppose that $\mc{K}_1\xra{t}\mc{K}_0$ is a relatively weakly mixing extension satisfying $p_0\circ t = p_1$.  Let $w:\mc{K}_0\rightarrow L$ be a measurable homomorphism into a Polish group $L$ admitting a bi-invariant metric, and suppose that the lifted homomorphism $w\circ t : \mc{K}_1 \rightarrow \mc{G} $ is equivalent to a homomorphism which descends through $p_1$ to $\mc{G}$. Then $w$ is equivalent to a homomorphism which descends through $p_0$ to $\mc{G}$.
\end{proposition}

\begin{proof}
By hypothesis there is a measurable map $\Phi :\mc{K}_1^0\rightarrow L$ and a homomorphism $u:\mc{G}\rightarrow L$ satisfying
\[
\Phi (r(k))w(t(k))\Phi (s(k))^{-1} = u(p_1(k)) = (u\circ p_0)(t(k))
\]
for a.e.\ $k\in \mc{K}_1$. Since $t$ is relatively weakly mixing, Lemma \ref{lem:Gbasic} implies that $\Phi = f\circ t$ for some map $f:\mc{K}_0^0\rightarrow L$. Hence $f(r(k_0))w(k_0)f(s(k_0))^{-1} = (u\circ p_0 )(k_0)$ for a.e. $k_0 \in \mc{K}_0$, and this shows that $w$ is equivalent to $u\circ p_0$.
\end{proof}

\begin{proposition}\label{prop:furman0}
Let $\mathscr{C}$ be a class of Polish groups contained in the class $\mathscr{G}_{\mathrm{inv}}$ of Polish groups admitting a bi-invariant metric. Let $\mathcal{H}$ be an aperiodic discrete p.m.p.\ groupoid which is Bernoulli $\mathscr{C}$-superrigid and let $K$ be any standard probability space. Then for every extension $\mathcal{G}\rightarrow\mathcal{H}$, the Bernoulli extension $p_0: \mathcal{G}\ltimes K^{\otimes \mathcal{G}} \rightarrow \mathcal{G}$ is relatively $\mathscr{C}$-superrigid.
\end{proposition}

\begin{proof}
Let $\mc{G}$ be an extension of $\mc{H}$, so that $\mc{G}$ is aperiodic as well. We have the following commutative diagram of groupoid extensions
\[
\xymatrix{
\mc{G}\ltimes [0,1]^{\otimes \mc{G}} \ar[r]^{q \ \ \ \ \ \ \ \, \, }_{\cong \ \ \ \ \ \ \ \, \, } \ar[d]^{p_1}  &\mc{G}\ltimes (K^{\otimes \mc{G}} \otimes _{\mc{G}^0} [0,1]^{\otimes \mc{G}}) \ar[d] \ar[r]^{\ \ \ \ \ \ \ \, \, \, p}  &\mc{G}\ltimes K^{\otimes \mc{G}}\ar[d]^{p_0} \\
\mc{G}\ar[r]^{\mathrm{id}_{\mc{G}}} &\mc{G}\ar[r]^{\mathrm{id}_{\mc{G}}}  &\mc{G}
}
\]
where $q$ is the isomorphism induced by an isomorphism $[0,1]\cong K\otimes [0,1]$ of probability spaces, and $p$ is the natural projection map. Since $\mc{H}$ is Bernoulli $\mathscr{C}$-superrigid, the extension $p_1$ is relatively $\mathscr{C}$-superrigid. Since $\mc{G}$ is aperiodic, the extension $p$ is relatively weakly mixing by Lemma \ref{lem:Bwm}, so the extension $t:= p\circ q$ is relatively weakly mixing as well, since $q$ is an isomorphism. Proposition \ref{prop:furman0relWM} therefore implies that $p_0$ is relatively $\mathscr{C}$-superrigid.
\end{proof}

\subsubsection{Pushing extensions to quotients} We now deal with a technical lemma used in the proof of Theorem \ref{thm:soe}.

Let $\mc{G}$ be a discrete p.m.p.\ groupoid. We let $[[\mc{G}]]$ denote the collection of all measurable bisections of $\mc{G}$, and we let $[\mc{G}]$ denote the collection of all measurable bisections $\theta$ of $\mc{G}$ with $s(\theta )=r(\theta ) = \mc{G}^0$. Then $[[\mc{G}]]$ and $[\mc{G}]$ are closed under the natural inverse and product operations defined on subsets of $\mc{G}$, and these operations make $[[\mc{G}]]$ an inverse semigroup, and $[\mc{G}]$ a group. Given a measurable bisection $\theta$ of $\mc{G}$, we denote by $c_{\theta} : \mc{G}_{s(\theta )} \rightarrow \mc{G}_{r(\theta )}$ the associated partial automorphism of $\mc{G}$ defined by $c_{\theta}(g) := \theta g \theta ^{-1}$ for $g\in \mc{G}_{s(\theta )}$.

Let $\mc{K}$ be a discrete p.m.p.\ groupoid, let $\mc{G}$ be a discrete Borel groupoid, and let $w$ and $v$ be two measurable homomorphisms from $\mc{K}$ into $\mc{G}$. We say that $w$ and $v$ are {\bf equivalent} if there exists a measurable map $\psi : \mc{K}^0\rightarrow \mc{G}$, with $\psi (x)\in \mc{G}w(x)$ for a.e.\ $x\in \mc{K}^0$, such that $\psi (r(k))w(k)\psi (s(k))^{-1} = v(k)$ for a.e.\ $k\in \mc{K}$. When $\mc{G}$ is a group then this definition is consistent with the definition of equivalence of homomorphisms into a group.

\begin{lemma}\label{lem:GOEcocycle}
Let $p: \mc{K}\rightarrow \mc{H}$ and $w:\mc{K}\rightarrow \mc{G}$ be extensions of ergodic discrete p.m.p.\ groupoids and assume that $\mc{H}$ is principal. Suppose that $w$ is equivalent to a groupoid homomorphism $v$ which descends to $\mc{H}$, i.e., $v=v_0\circ p$ for some groupoid homomorphism $v_0:\mc{H}\rightarrow \mc{G}$. Then there exists a bisection $\theta \in [\mc{K}]$, and a groupoid extension $w_0:\mc{H}\rightarrow \mc{G}$ such that $w\circ c_{\theta} = w_0\circ p$, and $w$ is equivalent to $w_0\circ p$.
\end{lemma}

\begin{proof}
Note that $\mc{K}$ is principal, being an extension of the principal groupoid $\mc{H}$. By assumption we may find a measurable map $\phi _0 :\mc{K}^0\rightarrow \mc{G}$ with $\phi _0(x) \in \mc{G}w(x)$ for $x\in \mc{K}^0$, such that
\[
\phi _0(r(k))w(k)\phi _0(s(k))^{-1} = v_0(p(k))
\]
for a.e.\ $k\in \mc{K}$. After discarding a null set we may assume without loss of generality that this holds everywhere, and we may also assume that both of the maps $p$ and $w$ are locally bijective. Define $\phi \subseteq \mc{K}$ by $\phi := \{ k\in \mc{K} \, : \,  w(k) = \phi _0 (s(k)) \}$. Then $\phi$ might not be a bisection, but $s : \phi \rightarrow \mc{K}^0$ is bijective (after discarding a null set), so for each $k\in \mc{K}$ the sets $\phi k$, $k\phi ^{-1}$, and $\phi k \phi ^{-1}$ contain exactly one element, and we abuse notation and denote these elements by $\phi k$, $k\phi ^{-1}$, and $\phi k \phi ^{-1}$ respectively. Then the map $c_{\phi}: \mc{K}\rightarrow \mc{K}$, $c_{\phi}(k):= \phi k \phi ^{-1}$ is a groupoid homomorphism. For each $k\in \mc{K}$ we have
\begin{equation}\label{eqn:GwT}
w(c_{\phi}(k))=w(\phi r(k))w(k)w(\phi s(k))^{-1} = \phi _0 (r(k))w(k)\phi _0 (s(k))^{-1} = v_0(p(k)) .
\end{equation}
Let $\mc{H}_1 = v_0^{-1}(\mc{G}^0)$ and let $\mc{K}_1 := (v_0\circ p)^{-1}(\mc{G}^0)$, so that $\mc{H}_1$ and $\mc{K}_1$ are subgroupoids of $\mc{H}$ and $\mc{K}$ respectively with $\mc{K}_1 = p^{-1}(\mc{H}_1)$. By \eqref{eqn:GwT}, for $k\in \mc{K}$ we have $(v_0\circ p )(k) \in \mc{G}^0$ if and only if $c_{\phi}(k)\in w^{-1}(\mc{G}^0)= \mc{K}^0$ and hence
\begin{equation}\label{eqn:GTx10}
\mc{K}_1 = c_{\phi}^{-1}(\mc{K}^0) .
\end{equation}

\medskip

\noindent {\bf Claim 0}. The groupoid $\mc{K}_1$ is periodic, i.e., $\mc{K}_1x$ is finite for a.e.\ $x\in \mc{K}^0$.

\begin{proof}[Proof of Claim 0]
For $k\in \mc{K}_1$ we have $c_{\phi}(k)\in \mc{K}^0$, hence
\begin{equation}\label{eqn:cphi}
c_{\phi }(k) = s(c_{\phi }(k)) = s(\phi k \phi ^{-1}) = r(\phi s(k)) = c_{\phi}(s(k)) .
\end{equation}
It follows that for each $x\in \mc{K}^0$ we have $\mc{K}_1x\subseteq c_{\phi}^{-1}(c_{\phi}(x))$. It is therefore enough to show that $c_{\phi}^{-1}(c_{\phi}(x))$ is finite for a.e.\ $x\in \mc{K}^0$, and since $(c_{\phi})_*\mu _{\mc{K}^0}$ is absolutely continuous with respect to $\mu _{\mc{K}^0}$, it is enough to show that $c_{\phi}^{-1}(x)$ is finite for a.e.\ $x\in \mc{K}^0$. Since $\mc{K}_1$ is principal, for each $x\in \mc{K}^0$, the map sending $k\in c_{\phi}^{-1}(x)$ to the pair $(r(k),s(k))$ is an injection from $c_{\phi}^{-1}(x)$ to a subset of the set of all pairs $(x_1,x_0 ) \in \mc{R}_{\mc{K}}$ for which $c_{\phi}(x_1)=x=c_{\phi}(x_0)$.  Therefore, it is enough to show that the set $c_{\phi}^{-1}(x)\cap \mc{K}^0$ is finite for a.e.\ $x\in \mc{K}^0$.

Let $D = \{ (x_1, x_0) \in \mc{R}_{\mc{K}}\, : \, c_{\phi}(x_1) = x_0  \}$. Then for each $x \in \mc{K}^0$ the vertical slice $D_{x}:= \{ x_1 \in \mc{K}^0 \, : \, (x_1,x) \in D \}$ has cardinality $|D_{x}| = |c_{\phi} ^{-1}(x)\cap \mc{K}^0|$, and the horizontal slice $D^{x} := \{ x_0 \in \mc{K}^0 \, : \, (x,x_0 ) \in D \}$ has cardinality $|D^{x}|=1$. Since $\mc{R}_{\mc{K}}$ is p.m.p.\ we have
\begin{align*}
\int _{\mc{K}^0} |c_{\phi}^{-1}(x)\cap \mc{K}_0 | \, d\mu _{\mc{K}^0} (x) = \int _{\mc{K}^0} |D_{x}| \, d\mu _{\mc{K}^0}(x) = \int _{\mc{K}^0} |D^x| \, d\mu _{\mc{K}^0}(x) = 1,
\end{align*}
and hence $c_{\phi}^{-1}(x)\cap \mc{K}^0$ is finite almost surely.
\qedhere[Claim 0]
\end{proof}

Since $\mc{K}_1 = p^{-1}(\mc{H}_1)$ and $p$ is measure preserving and locally bijective, Claim 0 implies that the groupoid $\mc{H}_1$ is periodic as well. Therefore, since $\mc{H}$ is principal, after discarding a null set we may find a Borel complete unit section $Y_0\subseteq \mc{H}^0$ for $\mc{H}_1$ such that $(\mc{H}_1)_{Y_0}=Y_0$. In particular $\mu _{\mc{H}}(Y_0) >0$. Then $X_0:=p^{-1}(Y_0)$ is a complete unit section for $\mc{K}_1$ with $(\mc{K}_1)_{X_0}=X_0$. Let $\theta _0 := \phi X_0$. Then $\mu _{\mc{K}}(\theta _0)>0$, and we claim that $\theta _0$ is a measurable bisection for $\mc{K}$. Indeed, since $\theta _0 \subseteq \phi$ we have that $s|\theta _0$ is injective, and $r|\theta _0$ is injective since if $k_0,k_1\in \theta _0$ satisfy $r(k_0)= r(k_1)$ then $c_{\phi}(k_0^{-1}k_1) = \phi (\phi s(k_0))^{-1}(\phi s(k_1))\phi ^{-1} = c_{\phi}(s(k_0))c_{\phi}(s(k_1)) \in \mc{K}^0$, so $k_0^{-1}k_1 \in (\mc{K}_1 )_{X_0}\subseteq \mc{K}^0$ and hence $k_0=k_1$.

Define now the sets
\begin{align*}
\mathscr{P}_1 &:= \{ \theta \in [[ \mc{K} ]] \, : \, s(\theta ) \text{ is }p\text{-measurable} \} , \\
\mathscr{P} &:= \{ \theta \in \mathscr{P}_1 \, : \, \text{the homomorphism } w\circ c_{\theta} : \mc{K}_{s(\theta )} \rightarrow \mc{G} \text{ is }p\text{-measurable} \} .
\end{align*}
For $\theta \in \mathscr{P}$ let $v_{\theta} : \mc{H}_{p(s(\theta ))} \rightarrow \mc{G}$ denote the homomorphism which lifts to $w\circ c_{\theta}$, so that $w\circ c_{\theta} = v_{\theta}\circ p$. We have $\theta _0 \in \mathscr{P}$ by \eqref{eqn:GwT}, with $s(\theta _0) = X_0$, $p(s(\theta _0)) = Y_0$, and $v_{\theta _0} = v_0|\mc{H}_{Y_0}$. We will show how to extend $\theta _0$ to some element of $\mathscr{P}$ with $\mu _{\mc{K}}(\theta )= 1$.

\medskip

\noindent {\bf Claim 1}. Let $\theta \in \mathscr{P}_1$, let $\tau \in [\mc{G}]$, and let $\tilde{\tau}  := w^{-1}(\tau )\in [\mc{K}]$. Then $v_{\theta}^{-1}(\tau )$ is a measurable bisection for $\mc{H}$, and $\theta ^{-1} \tilde{\tau}\theta =p^{-1}(v_{\theta}^{-1}(\tau ))$.
\begin{proof}[Proof of Claim 1]
We have
\[
\theta ^{-1}\tilde{\tau} \theta  = (w\circ c_{\theta} )^{-1}(\tau ) = (v_{\theta} \circ p)^{-1}(\tau ) = p^{-1}(v_{\theta}^{-1}(\tau )) ,
\]
and it remains to show that $v_{\theta}^{-1}(\tau )$ is a bisection for $\mc{H}$. Since the inverse image of $v_{\theta}^{-1}(\tau )$ under $p$ is a bisection for $\mc{K}$, this easily follows from local bijectivity of $p$.
\qedhere[Claim 1]
\end{proof}

\noindent {\bf Claim 2}. For every $\theta _1 \in \mathscr{P}$ with $0<\mu _{\mc{K}}(\theta _1) <1$ there exists some $\theta_2 \in \mathscr{P}$ with $\theta _1 \subseteq \theta _2$ and $\mu _{\mc{K}}(\theta _2) >\mu _{\mc{K}}(\theta _1)$.

\begin{proof}[Proof of Claim 2]
For $\sigma \in [[\mc{H}]]$ define $\tilde{\sigma} := p^{-1}(\sigma )\in [[\mc{K}]]$. By ergodicity of $\mc{K}$ we may find some $\sigma \in [[\mc{H}]]$ with $r(\tilde{\sigma }) \subseteq s(\theta _1 )$ such that $\theta _1 \tilde{\sigma}$ is non-null and disjoint from $\mc{K}s(\theta _1 )$. By ergodicity again, we may find some $\tau \in [\mc{G}]$ such that
\[
\theta _2 ' := (\tilde{\tau} \theta _1 \tilde{\sigma} )\setminus r(\theta _1)\mc{K}
\]
is non-null. Since $\theta _2'$ is a bisection of $\mc{K}$ which is disjoint from $r(\theta _1)\mc{K}s(\theta _1)$, the union $\theta _2 := \theta _1 \cup \theta _2 '$ is a bisection of $\mc{K}$, and clearly $\theta _1\subseteq \theta _2$ and $\mu _{\mc{K}}(\theta _1)<\mu _{\mc{K}}(\theta _2 )$.

We now show that $s(\theta _2 ' )$ is $p$-measurable. Using Claim 1, we have
\begin{align*}
s((\tilde{\tau} \theta _1 \tilde{\sigma} )\cap r(\theta _1)\mc{K}) 
&= s(\theta _1^{-1}\tilde{\tau} \theta _1 \tilde{\sigma}) = s(p^{-1}(v_{\theta _1}^{-1}(\tau )\sigma )) = p^{-1}(s(v_{\theta _1}^{-1}(\tau )\sigma )).
\end{align*}
Therefore,
\[
s(\theta _2 ' ) = s(\tilde{\sigma })\setminus p^{-1}(s(v_{\theta _1}^{-1}(\tau )\sigma )) = p^{-1}(s(\sigma ) \setminus s(v_{\theta _1}^{-1}(\tau )\sigma )),
\]
hence $s(\theta _2')$ is $p$-measurable. It follows that $s(\theta _2) = s(\theta _1 )\cup s(\theta _2')$ is $p$-measurable as well.

It remains to show that $w\circ c_{\theta _2} : \mc{K}_{s(\theta _2)}\rightarrow \mc{G}$ is $p$-measurable. For this, it is enough to show that the restriction of $w\circ c_{\theta _2}$ to each of the ($p$-measurable) sets $\mc{K}_{s(\theta _1)}$, $\mc{K}_{s(\theta _2')}$, $s(\theta _2')\mc{K} s(\theta _1 )$, and $s(\theta _1)\mc{K}s(\theta _2 ')$, is $p$-measurable. For the first set this is clear since $c_{\theta _2}$ coincides with $c_{\theta _1}$ on $\mc{K}_{s(\theta _1)}$. For the second set, on $\mc{K}_{s(\theta _2 ')}$ we have $c_{\theta _2} = c_{\theta _2 '} = c_{\tilde{\tau}}\circ c_{\theta _1}\circ c_{\tilde{\sigma}}$ and hence on this set we have the identity
\begin{align*}
w\circ c_{\theta _2} &= c_{\tau}\circ w \circ c_{\theta _1} \circ c_{\tilde{\sigma}} = c_{\tau}\circ v_{\theta _1}\circ p\circ c_{\tilde{\sigma}} = c_{\tau}\circ v_{\theta _1}\circ c_{\sigma}\circ p ,
\end{align*}
which is $p$-measurable. Similarly, for $k \in s(\theta _2')\mc{K} s(\theta _1 )$ we have
\begin{align*}
w(c_{\theta _2} (k)) &= w(\tilde{\tau}\theta _1 \tilde{\sigma}k\theta _1 ^{-1}) = \tau w(c_{\theta _1}(\tilde{\sigma}k)) = \tau v_{\theta _1}(p(\tilde{\sigma}k)) = \tau v_{\theta _1}(\sigma p(k)),
\end{align*}
which is $p$-measurable. Finally, the restriction $(w \circ c_{\theta _2})|s(\theta _1)\mc{K}s(\theta _2 ')$ is just the composition of the restriction $(w \circ c_{\theta _2})|s(\theta _2')\mc{K}s(\theta _1)$ with the inverse map $k\mapsto k^{-1}$, hence it is $p$-measurable as well.
\qedhere[Claim 2]
\end{proof}

By Claim 2 and measure theoretic exhaustion we may find some $\theta \in \mathscr{P}$ with $\theta _0\subseteq \theta$ and $\mu _{\mc{K}}(\theta ) =1$. After discarding a null set we may assume that $s(\theta ) = r(\theta )=\mc{K}^0$. We have $w\circ c_{\theta} =  v_{\theta}\circ p$, and so the homomorphism $v_{\theta}: \mc{H}\rightarrow \mc{G}$ is measure preserving since both $w\circ c_{\theta}$ and $p$ are measure preserving, i.e., $v_{\theta}$ is a groupoid extension. We put $w_0:= v_{\theta}$. Let $\psi :\mc{K}^0\rightarrow \mc{G}$ be the function $\psi (x) := w(\theta x )\in \mc{G}w(x)$. Then $\psi (r(k))w(k)\psi (s(k))^{-1} = (w\circ c_{\theta})(k)=(w_0\circ p)(k)$, hence $w$ is equivalent to $w_0\circ p$.
\end{proof}

\subsection{Groupoid Weak Pinsker entropy and the proof of Theorem \ref{thm:soe}}

\begin{theorem}\label{thm:gWPvsWP}
Let $G$ be a countable group which is Bernoulli $\mathscr{G}_{\mathrm{dsc}}$-superrigid. Let $G\cc X$ be any ergodic p.m.p.\ action of $G$ and let $G\ltimes X$ be the associated translation groupoid. Then $h^{\mathrm{gWP}} (G\ltimes X ) = h^{\mathrm{WP}}(G\cc X )$.
\end{theorem}

\begin{proof}
If $G$ is finite then both entropies are $0$, so we may assume that $G$ is infinite. Let $w:G\ltimes X \rightarrow G$ denote the projection to $G$. The inequality $h^{\mathrm{gWP}} (G\ltimes X) \geq h^{\mathrm{WP}}(G\cc X)$ always holds, so it suffices to show that $h^{\mathrm{gWP}} (G\ltimes X ) \leq h^{\mathrm{WP}}(G\cc X)$. If $h^{\mathrm{gWP}} (G\ltimes X )=0$ then this is trivial, so we may assume that $h^{\mathrm{gWP}} (G\ltimes X) >0$. Fix $r$ with $0<r <h^{\mathrm{gWP}}(G\ltimes X)$. It is enough to show that $h^{\mathrm{WP}}(G\cc X) > r$.

Since $h^{\mathrm{gWP}}(G\ltimes X )>r$, by Corollary \ref{cor:principal} we may find a principal ergodic discrete p.m.p.\ groupoid $\mc{H}$ and a probability space $K$ with $H(K)>r$, such that $G\ltimes X$ is isomorphic to the translation groupoid $\mc{H}\ltimes K^{\otimes \mc{H}}$, associated to the Bernoulli action of $\mc{H}$ with base $K$. Let $q:G\ltimes X \rightarrow \mc{H}\ltimes K^{\otimes \mc{H}}$ be an isomorphism, and let $p : G\ltimes X \rightarrow \mc{H}$ be the composition of $q$ with the projection $\mc{H}\ltimes K^{\otimes \mc{H}}\rightarrow \mc{H}$ to $\mc{H}$. Since $G\ltimes X$ is an extension of both $G$ and $\mc{H}$, the groupoids $G$ and $\mc{H}$ are measure equivalent, and hence $\mc{H}$ is Bernoulli $\mathscr{G}_{\mathrm{dsc}}$-superrigid by Theorem \ref{thm:Gmain}. Since the extension $G\ltimes X \xra{p}\mc{H}$ is isomorphic to a Bernoulli extension of $\mc{H}$, Proposition \ref{prop:furman0} shows that $p$ is relatively $\mathscr{G}_{\mathrm{dsc}}$-superrigid. Since $w$ takes values in the discrete group $G$, it follows that $w$ is cohomologous to a cocycle $v$ which descends through $p$ to $\mc{H}$. Applying Lemma \ref{lem:GOEcocycle}, we may find a bisection $\theta \in [G\ltimes X]$ and a groupoid extension $w_0 : \mc{H}\rightarrow G$ such that $w\circ c_{\theta}= w_0\circ p$.

We then have the following commutative diagram of extensions, where the top row consists of isomorphisms, and the bottom row is the identity map on $G$:
\[
\xymatrix{
G\ltimes X \ar[dd]_w \ar[r]^{c_{\theta}^{-1}}  &G\ltimes X \ar[dd]_{w\circ c_{\theta}} \ar[r]^{q \, \, \, \, \, \, }            & \mc{H}\ltimes K^{\otimes \mc{H}} \ar[d]\ar[r]^{\cong \, \, \, \, \, \, \, \, \, \, \, }     & \mc{H}\otimes _G (G\ltimes K^G ) \ar[d] \ar[r]^{\cong}    & G\ltimes (\mc{H}^0 \otimes K^G) \ar[dd] \\
                            &                                            &\mc{H}\ar[d]_{w_0}\ar[r]^{\mathrm{id}_{\mc{H}}}                     &\mc{H}\ar[d]_{w_0}                                         &    \\
G\ar[r]         &G\ar[r]                       &G\ar[r]                                      &G\ar[r]                                      & G
}
\]
In the top row, the third isomorphism is given by Proposition \ref{prop:BernIsom} applied to the extension $\mc{H}\xra{w_0}G$, and the rightmost isomorphism is obtained by composing the isomorphism $\mc{H}\otimes _G (G\ltimes K^G ) \cong (G\ltimes \mc{H}^0) \otimes _G (G\ltimes K^G )$ coming from Proposition \ref{prop:ExtActCorr} applied to $w_0$, with the isomorphism $(G\ltimes \mc{H}^0) \otimes _G (G\ltimes K^G )  \cong G\ltimes (\mc{H}^0\otimes K^G)$ coming from Remark \ref{rem:translext}. This shows that the $G$-extensions $G\ltimes X\xra{w} G$ and $G\ltimes (\mc{H}^0\otimes K^G )\rightarrow G$ are isomorphic, and hence the action $G\cc X$ is measurably conjugate to the product action $G\cc \mc{H}^0\otimes K^G$. Therefore, $h^{\mathrm{WP}}(G\cc X ) \geq H(K) >r$.\qedhere
\end{proof}

\begin{theorem}\label{thm:gWPred}
Let $\mc{G}$ be an aperiodic ergodic discrete p.m.p.\ groupoid and let $A$ be a positive measure subset of $\mc{G}^0$. Then $h^{\mathrm{gWP}}(\mc{G}_A) = \frac{1}{\mu _{\mc{G}}(A)}h^{\mathrm{gWP}}(\mc{G})$.
\end{theorem}

\begin{proof}
We first show $h^{\mathrm{gWP}}(\mc{G}_A) \leq \mu _{\mc{G}} (A)^{-1}h^{\mathrm{gWP}}(\mc{G})$. This is trivial if $h^{\mathrm{gWP}}(\mc{G}_A)=0$, so assume that $h^{\mathrm{gWP}}(\mc{G}_A) >0$. Given $r$ with $0<r<h^{\mathrm{gWP}}(\mc{G}_A)$, we will show that $r< \mu_{\mc{G}} (A)^{-1} h^{\mathrm{gWP}}(\mc{G})$. Since $0<r<h^{\mathrm{gWP}}(\mc{G}_A)$, by Corollary \ref{cor:principal} we may find a principal, ergodic discrete p.m.p.\ groupoid $\mc{H}_0$ and a probability space $K$ with $H(K)>r$ such that $\mc{G}_A$ is isomorphic to the translation groupoid $\mc{H}_0\ltimes K^{\otimes \mc{H}_0}$. By considering the $\mu _{\mc{G}}(A)^{-1}$-amplification of $\mc{H}_0$ (i.e., Proposition \ref{prop:amp} with $t=\mu _{\mc{G}}(A)^{-1}$), we may assume without loss of generality that $\mc{H}_0$ is the reduction $\mc{H}_0=\mc{H}_B$, of a principal ergodic discrete p.m.p.\ groupoid $\mc{H}$ to a subset $B$ of $\mc{H}^0$ with $\mu _{\mc{H}}(B)=\mu _{\mc{G}}(A)$. Let $q_0: \mc{H}_B\ltimes K^{\otimes \mc{H}_B}\rightarrow \mc{G}_A$ be an isomorphism.

Let $L$ be a probability space with $H(L)=\mu _{\mc{G}}(A)H(K)$. By \cite[Theorem 3.3]{BHI15}, there is an isomorphism of $\mc{H}_B$-extensions, call it $q_1$, from the reduction $(\mc{H}\ltimes L^{\otimes \mc{H}} )_B \rightarrow \mc{H}_B$ to the groupoid $\mc{H}_B\ltimes K^{\otimes \mc{H}_B}\rightarrow \mc{H}_B$. The sets $B$ and $A$ have the same measure, and the groupoids $\mc{H}\ltimes L^{\otimes \mc{H}}$ and $\mc{G}$ are both ergodic, so the isomorphism $q_0\circ q_1$, from $(\mc{H}\ltimes L^{\otimes \mc{H}})_B$ to $\mc{G}_A$, extends to an isomorphism from $\mc{H}\ltimes L^{\otimes \mc{H}}$ to $\mc{G}$ by Proposition \ref{prop:ExtIsoRed}. Therefore, $h^{\mathrm{gWP}}(\mc{G})\geq H(L)=\mu _{\mc{G}}(A) H(K) > \mu _{\mc{G}}(A)r$, as was to be shown.

We now show $h^{\mathrm{gWP}}(\mc{G})\leq \mu _{\mc{G}} (A)h^{\mathrm{gWP}}(\mc{G}_A)$. Once again we may assume $0<h^{\mathrm{gWP}}(\mc{G})$. Fix $0<r<h^{\mathrm{gWP}}(\mc{G})$ toward the goal of showing $r<\mu _{\mc{G}} (A)h^{\mathrm{gWP}}(\mc{G}_A)$. By Corollary \ref{cor:principal} we may find a principal ergodic discrete p.m.p.\ groupoid $\mc{H}$, a probability space $K$ with $H(K)>r$, and an isomorphism $q$ from $\mc{G}$ to the translation groupoid $\mc{H}\ltimes K^{\otimes \mc{H}}$. Let $p_{\mc{H}}:\mc{H}\ltimes K^{\otimes \mc{H}}\rightarrow \mc{H}$ be the projection map. Let $C$ be a subset of $\mc{H}^0$ with $\mu _{\mc{H}}(C)= \mu _{\mc{G}}(A)$. Then the set $B:= q^{-1}(p_{\mc{H}}^{-1}(C))\subseteq \mc{G}^0$ has the same measure as $A$, so since $\mc{G}$ is ergodic the groupoids $\mc{G}_A$ and $\mc{G}_B$ are isomorphic by Proposition \ref{prop:IsoRed}. Thus, $\mc{G}_A$ is isomorphic to the reduction $(\mc{H}\ltimes K^{\otimes \mc{H}})_{p_{\mc{H}}^{-1}(C)}$, and by \cite[Theorem 3.3]{BHI15} this is in turn isomorphic to the translation groupoid $\mc{H}_C \ltimes L^{\otimes \mc{H}_C}$ associated to the Bernoulli shift of $\mc{H}_C$ with base space entropy $H(L)=\mu _{\mc{H}}(C)^{-1}H(K)=\mu _{\mc{G}}(A)^{-1}H(K)$. This proves the inequality $h^{\mathrm{gWP}}(\mc{G}_A)\geq \mu _{\mc{G}}(A)^{-1}H(K)>\mu _{\mc{G}}(A)^{-1}r$.
\end{proof}

Theorem \ref{thm:soe} now follows easily from Theorems \ref{thm:gWPvsWP}, \ref{thm:gWPred}, and \ref{thm:main}.

\bibliographystyle{plain}
\bibliography{biblio-1}

\end{document}